\documentclass[a4paper,reqno]{amsart} 

\usepackage{amssymb}
\usepackage[mathcal]{euscript}

\usepackage{tikz-cd} 

\usepackage{hyperref}

\newcommand{\bydef}{:=}

\newcommand{\Skew}{\mathrm{Skew}}
\newcommand{\id}{\mathrm{id}}

\newcommand{\trace}{\mathrm{tr}}

\newcommand{\cA}{\mathcal{A}}
\newcommand{\cB}{\mathcal{B}} 
\newcommand{\cC}{\mathcal{C}}

\newcommand{\cK}{\mathcal{K}}
\newcommand{\cL}{\mathcal{L}}

\newcommand{\cQ}{\mathcal{Q}} 

\newcommand{\cS}{\mathcal{S}} 

\newcommand{\cU}{\mathcal{U}} 
\newcommand{\cV}{\mathcal{V}}

\newcommand{\frg}{{\mathfrak g}}

\newcommand{\frd}{{\mathfrak d}}

\newcommand{\ZZ}{\mathbb{Z}}

\newcommand{\FF}{\mathbb{F}} 

\DeclareMathOperator{\Hom}{\mathrm{Hom}}
\DeclareMathOperator{\End}{\mathrm{End}}

\DeclareMathOperator{\Der}{\mathrm{Der}}

\newcommand{\ad}{\mathrm{ad}}
\newcommand{\Ad}{\mathrm{Ad}}

\newcommand{\frsl}{{\mathfrak{sl}}}

\newcommand{\frso}{{\mathfrak{so}}}
\newcommand{\frpsl}{{\mathfrak{psl}}}
\newcommand{\frgl}{{\mathfrak{gl}}}
\newcommand{\frpgl}{{\mathfrak{pgl}}}
\newcommand{\frosp}{{\mathfrak{osp}}}

\newcommand{\tri}{\mathfrak{tri}}

\providecommand{\espan}[1]{\text{span}\left\{ #1\right\}}

\newcommand{\subo}{_{\bar 0}} 
\newcommand{\subuno}{_{\bar 1}}

\newcommand{\boldbup}{\textup{\textbf{b}}}
\newcommand{\nup}{\textup{n}}
\newcommand{\boldnup}{\textup{\textbf{n}}}

\newcommand{\Repe}{\mathsf{Rep\,C}_3}

\newcommand{\sVec}{\mathsf{sVec}}
\newcommand{\balpha}{\boldsymbol{\alpha}}
\newcommand{\Repa}{\mathsf{Rep}\,\balpha_3}
\newcommand{\Ver}{\mathsf{Ver}_3}

\newtheorem{theorem}{Theorem}[section]
\newtheorem{proposition}[theorem]{Proposition}
\newtheorem{lemma}[theorem]{Lemma}
\newtheorem{corollary}[theorem]{Corollary}

\newtheorem{properties}[theorem]{Properties}
\newtheorem{recipe}[theorem]{Recipe}

\theoremstyle{definition}

\theoremstyle{remark} \newtheorem{remark}[theorem]{Remark}
\numberwithin{equation}{section}

\def\bigstrut{\vrule height 12pt width 0ptdepth 2pt}

\def\hregleta{\hrule height .5pt}
\def\hreglon{\hrule height 1pt}
\def\vreglon{\vrule height 12pt width1pt depth 4pt}
\def\vregleta{\vrule width .5pt}
\def\hreglonfill{\leaders\hreglon\hfill}

\def\hregletafill{\leaders\hregleta\hfill}

\begin{document}

\title{From octonions to composition superalgebras 
via tensor categories}

\author[A.~Daza-Garc{\'\i}a]{Alberto Daza-Garc{\'\i}a}
\address[A.\,D., A.\,E.]{Departamento de
Matem\'{a}ticas e Instituto Universitario de Matem\'aticas y
Aplicaciones, Universidad de Zaragoza, 50009 Zaragoza, Spain}
\email{albertodg@unizar.es, elduque@unizar.es} 
\thanks{The first two authors are supported by grant
MTM2017-83506-C2-1-P (AEI/FEDER, UE), 
by grant PID2021-123461NB-C21, funded by MCIN/AEI/10.13039/501100011033 and by
 ``ERDF A way of making Europe'',
 and by grant E22\_20R (Gobierno de 
Arag\'on, Grupo de investigaci\'on ``\'Algebra
y Geometr{\'\i}a'').\\
The first author also acknowledges support by the F.P.I. 
grant PRE2018-087018.}

\author[A.~Elduque]{Alberto Elduque} 

\author[U.~Sayin]{Umut Sayin}
\address[U.~S]{Department of Mathematics, D\"uzce University,
81620 Konuralp, D\"uzce, Turkey}
\email{umutsayin@duzce.edu.tr }
\thanks{The third author is supported by grant T\"UB{\.I}TAK 2219}




\begin{abstract}
The nontrivial unital composition superalgebras, of dimension $3$ and $6$, which 
exist only in characteristic $3$, are obtained from the split Cayley algebra and its order $3$ automorphisms, by means of the process of semisimplification of the symmetric tensor category of representations of the cyclic group of order $3$.
Connections with the extended Freudenthal Magic Square in characteristic $3$, that
contains some exceptional Lie superalgebras specific of this characteristic are discussed too.

In the process, precise recipes to go from (nonassociative) algebras in this tensor category to the corresponding superalgebras are given.
\end{abstract}

\maketitle

\section{Introduction}\label{se:intro}

In \cite{Eld06}, Lie algebras $\frg$ over a field $\FF$ that admit a 
$\ZZ/2$-grading such that the even part is the direct sum of $\frsl_2(\FF)$ 
and another ideal $\frd$, and its odd part is, as a module for the even part, a tensor product of the two-dimensional natural module for $\frsl_2(\FF)$ and a module $T$ for $\frd$, were considered. Thus, we have 
\begin{equation}\label{eq:STS}
\frg=\bigl(\frsl_2(\FF)\oplus\frd\bigr)\oplus \bigl(\FF^2\otimes T).
\end{equation}
In this case, $T$ becomes a so-called \emph{symplectic triple system}, and the invariance of the Lie bracket under the action of $\frsl_2(\FF)$ forces the bracket of odd elements to present the following form:
\[
[u\otimes x,v\otimes y]=(x\mid y)\gamma_{u,v}+\langle u\mid v\rangle d_{x,y}
\]
for all $u,v\in \FF^2$ and $x,y\in T$, for a skew-symmetric bilinear form $(\cdot\mid\cdot)$ on $T$ and a symmetric bilinear map $T\times T\rightarrow \frd$, 
$(x,y)\mapsto d_{x,y}$; where $\langle u\mid v\rangle$ is, up to scalars, the 
unique $\frsl_2(\FF)$-invariant bilinear form on $\FF^2$, and 
$\gamma_{u,v}=\langle u\mid \cdot \rangle v+\langle v\mid\cdot \rangle u$.

All classical simple Lie algebras can be obtained in this way.

But the main point raised in \cite{Eld06} was that, in case the characteristic of 
$\FF$ is $3$, then the $\ZZ/2$-graded vector space $\frd\oplus T$, with bracket given by the bracket in $\frd$, the action of $\frd$ in $T$, and by 
$[x,y]=d_{x,y}$ for 
$x,y\in T$, endows $\frd\oplus T$ with a structure of Lie superalgebra. This remark gave the construction of a family of new simple contragredient simple Lie superalgebras specific of characteristic $3$. Another family of such simple 
Lie superalgebras was obtained in \cite{Eld06} by means of simple orthogonal 
triple systems, and most of these new simple Lie superalgebras appeared in a unified way in the Extended Freudenthal Magic Square in
\cite{CunhaElduque}. (See also \cite{BGL}.)

\smallskip

Quite recently \cite{Kannan}, Arun S.~Kannan considered a much more general and surprising way of passing from Lie algebras to Lie superalgebras, obtaining the simple Lie superalgebras mentioned above in a quite combinatorial way. Another exceptional Lie superalgebra specific of characteristic $5$, first obtained in \cite{Eld07}, is obtained too by Kannan, using a variation of his method in characteristic $5$.

Kannan considered, over fields of characteristic $3$, exceptional simple Lie algebras endowed with a nilpotent derivation $d$ with $d^3=0$. This allows to 
view the Lie algebra as a Lie algebra in the category $\Repa$ of representations of the affine group scheme $\balpha_3: R\mapsto \{r\in R\mid r^3=0\}$ (the kernel of the Frobenius endomorphism of the additive group scheme $\mathbb{G}_a$). The \emph{semisimplification} of $\Repa$ is the Verlinde category $\Ver$, which is equivalent to the category of vector superspaces.  In this way, a path is obtained from Lie algebras in $\Repa$ to Lie superalgebras.

For Lie algebras as in \eqref{eq:STS}, we may choose $d$ to be the adjoint action by $\left(\begin{smallmatrix} 0&1\\ 0&0\end{smallmatrix}\right)$. In this case, the ideal 
$\frsl_2(\FF)$ constitutes a Jordan block of length $3$ for $d$. The ideal $\frd$ is annihilated by $d$, and the odd part $\FF^2\otimes T$ is a direct sum of Jordan blocks of length $2$, as $d$ is nilpotent of order $2$ on $\FF^2$. The semisimplification process in \cite{Kannan} returns precisely the Lie superalgebra 
$\frd\oplus T$ above.

\smallskip

In this paper, we want to concentrate on another feature in characteristic $3$. 
Only over fields of this characteristic there are nontrivial 
\emph{composition superalgebras} (see \cite{EldOkubo}). 
Our goal is to obtain the two unital composition superalgebras with nontrivial odd part: $B(1,2)$ and $B(4,2)$, from the split Cayley algebra by the process 
of semisimplification. It must be remarked that these 
composition superalgebras appeared for the first time in Shestakov's work on 
prime alternative superalgebras \cite{Shestakov}.
Actually, we will not semisimplify from $\Repa$ as in \cite{Kannan}, 
but from the category $\Repe$ 
of representations of the cyclic group of order $3$ (or equivalenty, from the category of representations of the constant group scheme $\mathsf{C}_3$). 
In other words, instead of considering algebras with a nilpotent derivation $d$ 
with $d^3=0$, we consider algebras endowed with an automorphism of order $3$. 
The semisimplification of $\Repe$ is again the Verlinde category $\Ver$.

\smallskip

The paper is organized as follows. Section \ref{se:alg_super} will review the needed results from the categories mentioned above. Our basic reference for monoidal and tensor categories will be \cite{EGNO}. Concise recipes will be given to describe 
the superalgebra obtained from an algebra in $\Repe$ by semisimplification. 
Section \ref{se:octonions} will be devoted to considering composition algebras in a symmetric tensor category, to reviewing the known results on order $3$ automorphisms of the Cayley algebras over fields of charactetistic $3$, and to using the recipes in the previous section in order to obtain $B(1,2)$ and $B(4,2)$ from the split Cayley algebra.  Section \ref{se:ss} will be devoted to showing how this process of semisimplification behaves with respect to algebras of derivations or of skew transformations relative to a nondegenerate symmetric bilinear form.  Also, the Extended Freudenthal Magic Square in \cite{CunhaElduque} is built in terms of composition superalgebras, and it will be shown in the last section how the work in Section \ref{se:octonions} can be used to obtain the Lie superalgebras in the extended square  by semisimplification from the algebras in the last row of the classical Freudenthal Magic Square, in a way different from the one considered in \cite{Kannan}.  That is, semisimplification provides a bridge between the classical Freudenthal Magic Square and its extended version.

\smallskip

Throughout the paper, $\FF$ will denote a ground field. All vector spaces will be assumed to be finite-dimensional over $\FF$ and unadorned tensor products will be over $\FF$. Most of the time, the characteristic of $\FF$ will be $3$.

\section{From algebras to superalgebras}\label{se:alg_super}

This section will review, in a way suitable for our purposes, known results on the
categories $\Repe$, $\Ver$, and $\sVec$. For details, the reader may consult
\cite{EK21,Ost15,EO19} and references there in.

Throughout this section, the characteristic of the ground field $\FF$ will always be $3$.

\subsection{Semisimplification of \texorpdfstring{$\Repe$}{RepC3}}\label{ss:semiRep}

The category $\Repe$, whose objects are the finite-dimensional representations of the finite group $\mathsf{C}_3$ over $\FF$  or, equivalently, of the corresponding
constant group scheme, and whose morphisms are the equivariant homomorphisms, is a symmetric tensor category, with the usual tensor product of vector spaces and the 
braiding given by the usual swap: $X\otimes Y\rightarrow Y\otimes X$, 
$x\otimes y\mapsto y\otimes x$. 

Fix a generator $\sigma$ of $\mathsf{C}_3$. $\Repe$ is not a semisimple category. The indecomposable objects are, up to isomorphism, $V_0=\FF$, $V_1=\FF v_0+\FF v_1$, and $V_2=\FF w_0+\FF w_1+\FF w_2$, where the action of $\sigma$ is trivial on $V_0$, 
and $\sigma(v_0)=v_0+v_1$, $\sigma(v_1)=v_1$; $\sigma(w_0)=w_0+w_1$, 
$\sigma(w_1)=w_1+w_2$, $\sigma(w_2)=w_2$. Any object $\cA$ in $\Repe$ decomposes,
nonuniquely, as
\begin{equation}\label{eq:A0A1A2}
\cA=\cA_0\oplus \cA_1\oplus \cA_2,
\end{equation}
where $\cA_i$ is a direct sum of copies of $V_i$, $i=0,1,2$.

A homomorphism $f\in\Hom_{\Repe}(X,Y)$ is said to be \emph{negligible} if for all homomorphisms $g\in\Hom_{\Repe}(Y,X)$, $\trace(fg)=0$ holds. Denote by 
$\mathrm{N}(X,Y)$ the subspace of negligible homomorphisms in 
$\Hom_{\Repe}(X,Y)$. 

For instance, $\End_{\Repe}(V_1)$ consists of those endomorphisms of $V_1$ which commute with
$\sigma$. Any such endomorphism $f$ satisfies $f(v_0)=\alpha v_0+\beta v_1$ and 
$f(v_1)=\alpha v_1$ for scalars $\alpha,\beta\in\FF$, so that $f=\alpha\id_{V_1}+g$ for a nilpotent endomorphism $g$. It follows that $\mathrm{N}(V_1,V_1)$ consists of the nilpotent endomorphisms
in $\End_{\Repe}(V_1)$. For $V_2$, any $f\in \End_{\Repe}(V_2)$ is again of the form 
$\alpha\id_{V_2}+g$ for a nilpotent endomorphism, but now $\trace(\id_{V_2})=0$ as the characteristic of $\FF$ is $3$, and it turns out that $\End_{\Repe}(V_2)$ consists 
entirely of negligible endomorphisms.

\smallskip

Negligible homomorphisms form a \emph{tensor ideal} and this allows us to define the 
\emph{semisimplification} of $\Repe$, which is the Verlinde category $\Ver$, whose objects are the objects of $\Repe$, but whose morphisms are given by
\[
\Hom_{\Ver}(X,Y)\bydef \Hom_{\Repe}(X,Y)/\mathrm{N}(X,Y).
\]
This is again a symmetric tensor category, with the tensor product in $\Repe$, and the braiding induced by the one in $\Repe$.

Denote by $[f]$ the class of $f\in\Hom_{\Repe}(X,Y)$ modulo 
$\mathrm{N}(X,Y)$. Note that the identity morphism in $\End_{\Ver}(X)$ is 
$[\id_{X}]$, where $\id_{X}$ denotes the identity morphism in $\Repe$ (the identity map). We have thus obtained the \emph{semisimplification} functor:
\begin{equation}\label{eq:S}
\begin{split}
S:\Repe&\longrightarrow \Ver\\
 X&\mapsto X \ \text{for objects,}\\
 f&\mapsto [f]\ \text{for morphisms.}
\end{split}
\end{equation}
The semisimplification functor $S$ is $\FF$-linear and braided monoidal (see \cite[Definitions 1.2.3 and 8.1.7]{EGNO}).

Some straightforward consequences of the definitions are recalled here:

\begin{properties}\label{properties}
\begin{itemize}
\item $\End_{\Ver}(V_i)=\FF[\id_{V_i}]\neq 0$ for $i=0,1$, $\End_{\Ver}(V_2)=0$,
$\Hom_{\Ver}(V_i,V_j)=0$ for $i\neq j$.

\item $V_0$ and $V_1$ are simple objects in $\Ver$, while $V_2$ is isomorphic to $0$.

\item $\Ver$ is semisimple: any object is isomorphic to a direct sum of copies of $V_0$ and $V_1$.

\item $V_0\otimes V_i$ and $V_i\otimes V_0$ are isomorphic to $V_i$, for $i=0,1$, both in $\Repe$ and in 
$\Ver$; while $V_1\otimes V_1$ is isomorphic to $V_0$ in 
$\Ver$.

Actually, in $\Repe$, $V_1\otimes V_1$ is the direct sum of its submodule of symmetric tensors, which is isomorphic to $V_2$, and its one-dimensional submodule of skew-symmetric tensors, which is isomorphic to $V_0$. An explicit isomorphism 
$V_1\otimes V_1\rightarrow V_0$ in $\Ver$ is $[\lambda]$ where 
$\lambda:V_1\otimes V_1\rightarrow V_0$ is the homomorphism in $\Repe$ given by:
\begin{equation}\label{eq:lambda}
v_0\otimes v_0\mapsto 0,\ v_1\otimes v_1\mapsto 0,\ 
v_0\otimes v_1\mapsto 1,\ v_1\otimes v_0\mapsto -1.
\end{equation}
Its inverse is $[\lambda']$, where $\lambda':V_0\rightarrow V_1\otimes V_1$ is
the homomorphism in $\Repe$ that takes $1$ to 
$\frac{1}{2}(v_0\otimes v_1-v_1\otimes v_0)$.

\item The braiding in $\Ver$, for objects $X,Y$, is given by $[c_{X,Y}]$, 
where $c_{X,Y}$ is the braiding in $\Repe$ (i.e., the swap $x\otimes y\mapsto y\otimes x$). Then, identifying 
$V_0\otimes V_0\simeq V_0$,
$V_0\otimes V_1\simeq V_1\simeq V_1\otimes V_0$, and $V_1\otimes V_1\simeq V_0$ as above, we have $[c_{V_0,V_0}]=[\id_{V_0}]$, 
$[c_{V_0,V_1}]=[\id_{V_1}]=[c_{V_1,V_0}]$, but  $[c_{V_1,V_1}]=-[\id_{V_0}]$, because 
we have
\[
c_{V_1,V_1}\left(\frac{1}{2}(v_0\otimes v_1-v_1\otimes v_0)\right)=
\frac{1}{2}(v_1\otimes v_0-v_0\otimes v_1)
=-\frac{1}{2}(v_0\otimes v_1-v_1\otimes v_0).
\]

\end{itemize}
\end{properties}

\bigskip

\subsection{Equivalence of \texorpdfstring{$\Ver$}{Ver} and 
\texorpdfstring{$\sVec$}{sVec}}\label{ss:Ver_sVec}

This equivalence is well known, but  concrete formulas for these equivalence will be
needed later on, and hence this will be reviewed in some detail.

The objects of the category $\sVec$ of vector superspaces (over our ground field 
$\FF$) are the $\ZZ/2$-graded vector spaces $X=X\subo\oplus X\subuno$, and the morphisms $f:X\rightarrow Y$ are the linear maps preserving this grading: $f(X\subo)\subseteq Y\subo$, $f(X\subuno)\subseteq Y\subuno$. We will write 
$f=f\subo\oplus f\subuno$, with $f_{\bar a}:X_{\bar a}\rightarrow Y_{\bar a}$ given by the restriction of $f$, $a=0,1$. This is a symmetric tensor category, with the braiding given by the \emph{parity swap}: 
\[
X\otimes Y\rightarrow Y\otimes X, \quad
x\otimes y\mapsto (-1)^{xy}y\otimes x,
\] 
for homogeneous elements $x,y$, where $(-1)^{xy}$ is 
$-1$ if both $x$ and $y$ are odd, and it is $1$ otherwise.

The $\FF$-linear functor given on objects and morphisms by
\begin{equation}\label{eq:F}
\begin{split}
F:\sVec&\longrightarrow \Ver\\
X\subo\oplus X\subuno &\mapsto X\subo\oplus(X\subuno\otimes V_1)\\
f\subo\oplus f\subuno&\mapsto [f\subo\oplus(f\subuno\otimes\id_{V_1})],
\end{split}
\end{equation}
is an equivalence of categories. Here the action of $\mathsf{C}_3$ on 
$X\subo\oplus (X\subuno\otimes V_1)$ is given by the action on $V_1$. That is, 
$X\subo$ is a trivial module for $\mathsf{C}_3$, while 
$\sigma(x\subuno\otimes v)\bydef x\subuno\otimes \sigma(v)$, for all 
$x\subuno\in X\subuno$ and $v\in V_1$.

$F$ is a monoidal functor with natural isomorphism 
$J:F(\cdot)\otimes F(\cdot)\rightarrow F(\cdot\otimes\cdot)$ given by 
$J_{X,Y}=[j_{X,Y}]$, where $j_{X,Y}$ is the morphism in $\Repe$ defined as follows,
for $X=X\subo\oplus X\subuno$ and $Y=Y\subo\oplus Y\subuno$:
\begin{equation}\label{eq:J}
\begin{split}
j_{X,Y}:\Bigl(X\subo\oplus(X\subuno\otimes V_1)\Bigr)\otimes 
\Bigl(Y\subo\oplus(Y\subuno\otimes V_1)\Bigr)&\longrightarrow\\
&\hspace*{-60pt}(X\subo\otimes Y\subo\oplus X\subuno\otimes Y\subuno)\oplus
\bigl((X\subo\otimes Y\subuno\oplus X\subuno\otimes Y\subo)\otimes V_1\bigr)\\
x\subo\otimes y\subo&\mapsto x\subo\otimes y\subo,\\
x\subo\otimes (y\subuno\otimes v)&\mapsto (x\subo\otimes y\subuno)\otimes v,\\
(x\subuno\otimes v)\otimes y\subo&\mapsto (x\subuno\otimes y\subo)\otimes v,\\
(x\subuno\otimes u)\otimes (y\subuno\otimes v)&\mapsto 
\lambda(u\otimes v)x\subuno\otimes y\subuno,
\end{split}
\end{equation}
for all $x\subo\in X\subo$, $x\subuno\in X\subuno$, $y\subo\in Y\subo$, $y\subuno\in Y\subuno$, and $u,v\in V_1$, where $\lambda$ is given in \eqref{eq:lambda}.

The inverse of $J_{X,Y}$ is $J_{X,Y}^{-1}=[j'_{X,Y}]$, where $j'_{X,Y}$ is defined as follows:
\begin{equation}\label{eq:J'}
\begin{split}
j'_{X,Y}:(X\subo\otimes Y\subo\oplus X\subuno\otimes Y\subuno)&\oplus
\bigl((X\subo\otimes Y\subuno\oplus X\subuno\otimes Y\subo)\otimes V_1\bigr)\\
&\longrightarrow
\Bigl(X\subo\oplus(X\subuno\otimes V_1)\Bigr)\otimes 
\Bigl(Y\subo\oplus(Y\subuno\otimes V_1)\Bigr)
\\
x\subo\otimes y\subo&\mapsto x\subo\otimes y\subo,\\
(x\subo\otimes y\subuno)\otimes v &\mapsto x\subo\otimes (y\subuno\otimes v),\\
(x\subuno\otimes y\subo)\otimes v&\mapsto (x\subuno\otimes v)\otimes y\subo,\\
x\subuno\otimes y\subuno&\mapsto 
\frac{1}{2}\Bigl((x\subuno\otimes v_0)\otimes(y\subuno\otimes v_1)-
(x\subuno\otimes v_1)\otimes(y\subuno\otimes v_0)\Bigr).
\end{split}
\end{equation}
Note that $F$ preserves the braiding too. In other words, the following diagram is commutative for all $X,Y$:
\[
\begin{tikzcd}
F(X)\otimes F(Y)\arrow[rr, "\hbox{$[\text{`swap'}]$}"]\arrow[d, "J_{X,Y}"']&&
F(Y)\otimes F(X)\arrow[d, "J_{Y,X}"]\\
F(X\otimes Y)\arrow[rr, "F(\text{`parity swap'})"]&&F(Y\otimes X)
\end{tikzcd}
\]
Therefore, the functor $F$ is a braided monoidal equivalence.

\bigskip

\subsection{Recipe to get superalgebras from algebras in 
\texorpdfstring{$\Repe$}{RepC3}}

Given a linear map $m:A\otimes B\rightarrow C$ in $\sVec$, the composition
\[
F(A)\otimes F(B)\xrightarrow{J_{A,B}}F(A\otimes B)\xrightarrow{F(m)} F(C)
\]
gives a homomorphism $F(A)\otimes F(B)\rightarrow F(C)$ in $\Ver$. In particular, with $A=B=C$, given an algebra $(A,m)$ in $\sVec$ (that is, a superalgebra $A$ with product a morphism
$m:A\otimes A\rightarrow A$, $x\otimes y\mapsto x\bullet y$, in $\sVec$), $F(A)$ is an algebra in $\Ver$ with multiplication given by the composition
\[
F(A)\otimes F(A)\xrightarrow{J_{A,A}}F(A\otimes A)\xrightarrow{F(m)} F(A).
\]

Now, given a homomorphism $\mu:\cA\otimes \cB\rightarrow \cC$ in $\Repe$, our goal
is to find explicitly objects $A,B,C$ in $\sVec$ and a homomorphism 
$m:A\otimes B\rightarrow C$ such that there are isomorphisms 
$[\iota_A]:F(A)\rightarrow \cA$, $[\iota_B]:F(B)\rightarrow \cB$, 
$[\iota_C]:F(C)\rightarrow \cC$ in $\Ver$ that make the following diagram commutative:
\begin{equation}\label{eq:FABCmuFm}
\begin{tikzcd}
F(A)\otimes F(B)\arrow[r, "J_{A,B}"]\arrow[d, "\hbox{$[\iota_A\otimes\iota_B]$}"']&
F(A\otimes B)\arrow[r, "F(m)"]&F(C)\arrow[d, "\hbox{$[\iota_C]$}"]\\
\cA\otimes\cB\arrow[rr, "\hbox{$[\mu]$}"]&&\cC
\end{tikzcd}
\end{equation}
In particular, given  an algebra $\cA$ in $\Repe$, with multiplication $\mu(x\otimes y)=xy$, our goal is to find explicitly the superalgebra $(A,m)$, unique up to isomorphism, such that the algebras 
$\bigl(F(A),F(m)\circ J_{A,A}\bigr)$ and $(\cA,[\mu])$ are isomorphic algebras in 
$\Ver$. This is achieved in  Corollary \ref{co:main}.

To begin with, note that the objects $\cA,\cB,\cC$ in $\Repe$ decompose as in \eqref{eq:A0A1A2}: $\cA=\cA_0\oplus\cA_1\oplus \cA_2$, 
$\cB=\cB_0\oplus\cB_1\oplus \cB_2$, and $\cC=\cC_0\oplus\cC_1\oplus \cC_2$, 
where $\cA_i$, $\cB_i$ and $\cC_i$ are direct sums of copies of $V_i$, $i=0,1,2$.
Write $\cA'=\cA_0\oplus \cA_1$, $\cB'=\cB_0\oplus\cB_1$, and 
$\cC'=\cC_0\oplus\cC_1$. Then Properties \ref{properties} immediately imply the following result:

\begin{lemma}\label{le:mu'}
Given objects $\cA,\cB,\cC$ in $\Repe$ with the above decompositions, let $\mu:\cA\otimes\cB\rightarrow \cC$ be a homomorphism in $\Repe$. Then the inclusion maps
$\cA'\hookrightarrow\cA$, $\cB'\hookrightarrow\cB$, and $\cC'\hookrightarrow\cC$, induce isomorphisms in $\Ver$, and the diagram
\[
\begin{tikzcd}
\cA'\otimes\cB'\arrow[rr, "\hbox{$[\mu']$}"]
\arrow[d, hook]&& \cC'\arrow[d, hook]\\
\cA\otimes\cB\arrow[rr, "\hbox{$[\mu]$}"]&&\cC
\end{tikzcd}
\]
is commutative, where $\mu'\in\Hom_{\Repe}(\cA'\otimes\cB',\cC')$ is given by the formula
\[
\mu'(x\otimes y)\bydef\begin{cases}
\mathrm{proj}_{\cC_0} \mu(x\otimes y)&\text{for $x\in\cA_0, y\in\cB_0$ or $x\in\cA_1, y\in\cB_1$,}\\
\mathrm{proj}_{\cC_1} \mu(x\otimes y)&\text{for $x\in\cA_0,y\in\cB_1$, or $x\in\cA_1,y\in\cB_0$.}
\end{cases}
\]
(The projections are relative to the splitting $\cC=\cC_0\oplus\cC_1\oplus\cC_2$.)
\end{lemma}

In particular,  if $(\cA,\mu)$ is an algebra in $\Repe$ 
(this means that $\sigma$ acts as an algebra automorphism), then the previous 
 lemma restricts as follows:

\begin{corollary}\label{co:mu'}
Let $(\cA,\mu)$ be an algebra in $\Repe$, with $\mu(x\otimes y)=xy$ for all 
$x,y\in\cA$. Pick a splitting as in \eqref{eq:A0A1A2}.
Then the algebra $(\cA,[\mu])$ in $\Ver$ is isomorphic to the algebra 
$(\cA',[\mu'])$, where $\cA'=\cA_0\oplus \cA_1$ and 
$\mu'\in\Hom_{\Repe}(\cA'\otimes\cA',\cA')$ is given by the formula
\[
\mu'(x\otimes y)=x\cdot y\bydef\begin{cases}
\mathrm{proj}_{\cA_0} xy&\text{for $x,y\in\cA_0$ or $x,y\in\cA_1$,}\\
\mathrm{proj}_{\cA_1} xy&\text{for $x\in\cA_0,y\in\cA_1$, or $x\in\cA_1,y\in\cA_0$.}
\end{cases}
\]
(The projections are relative to the splitting $\cA=\cA_0\oplus\cA_1\oplus\cA_2$.)
\end{corollary}

\begin{remark}\label{re:delta}
Let $\mu:\cA\otimes\cB\rightarrow\cC$ be a homomorphism in $\Repe$ as before. On each of $\cA,\cB,\cC$, let the endomorphism $\delta$ in $\Repe$  be defined by 
$\delta(x)=\sigma(x)-x$. Let $\mu':\cA'\otimes\cB'\rightarrow\cC'$ be defined as in Lemma \ref{le:mu'}. Then, for any $x\in\cA_1$ and $y\in\cB_1$, the following 
equation holds:
\begin{equation}\label{eq:mu'xdeltay}
\mu'\bigl(x\otimes\delta(y)\bigr)=-\mu'\bigl(\delta(x)\otimes y\bigr)=
\frac{1}{2}\Bigl(\mu'\bigl(x\otimes\delta(y)\bigr)
  -\mu'\bigl(\delta(x)\otimes y\bigr)\Bigr).
\end{equation}
Indeed, write $\mu'(x\otimes y)=x\cdot y$. Let $c=x\cdot\delta(y)$, which belongs to 
$\cC_0$. Then we get
\[
c=\sigma(c)=\sigma(x)\cdot\sigma(\delta(y))=(x+\delta(x))\cdot \delta(y)=
c+\delta(x)\cdot\delta(y),
\]
so that $\delta(x)\cdot\delta(y)=0$ holds. Now write $d=x\cdot y\in \cA_0$. We get
\begin{multline*}
d=\sigma(d)=\sigma(x)\cdot\sigma(y)=(x+\delta(x))\cdot(y+\delta(y))\\=
d+x\cdot\delta(y)+\delta(x)\cdot y+\delta(x)\cdot\delta(y)=
d+x\cdot\delta(y)+\delta(x)\cdot y,
\end{multline*}
and \eqref{eq:mu'xdeltay} follows.
\end{remark}

Let $\mu:\cA\otimes\cB\rightarrow\cC$ be a homomorphism in $\Repe$ as in Remark \ref{re:delta}. Fix splittings of $\cA,\cB,\cC$ as in \eqref{eq:A0A1A2}, and pick subspaces $A\subuno$ of $\cA_1$ (resp., $B\subuno$ of $\cB_1$, $C\subuno$ of $\cC_1$) such that $\cA_1=A\subuno\oplus\delta(A\subuno)$ (resp., 
$\cB_1=B\subuno\oplus\delta(B\subuno)$, $\cC_1=C\subuno\oplus\delta(C\subuno)$), where, as before, $\delta=\sigma-\id$. Write $A\subo=\cA_0$ (resp., $B\subo=\cB_0$,
$C\subo=\cC_0$). Then $\cC$ decomposes as
\begin{equation}\label{eq:C0C1C1C2}
\cC=C\subo\oplus C\subuno\oplus\delta(C\subuno)\oplus\cC_2,
\end{equation}
and similarly for $\cA$ and $\cB$. Consider the objects $A=A\subo\oplus A\subuno$,
$B=B\subo\oplus B\subuno$, and $C=C\subo\oplus C\subuno$ in $\sVec$.

\begin{recipe}\label{recipe1} Take projections relative to the splitting
\eqref{eq:C0C1C1C2}, and define the homomorphism $m:A\otimes B\rightarrow C$ in
$\sVec$ as follows:
\[
\begin{split}
m(x\subo\otimes y\subo) &=\mathrm{proj}_{C\subo} \mu(x\subo\otimes y\subo)\\
m(x\subo\otimes y\subuno)&=\mathrm{proj}_{C\subuno} \mu(x\subo\otimes y\subuno)\\
m(x\subuno\otimes y\subo)&=\mathrm{proj}_{C\subuno} \mu(x\subuno\otimes y\subo)\\
m(x\subuno\otimes y\subuno)&=\mathrm{proj}_{C\subo} 
                    \mu\bigl(x\subuno\otimes\delta(y\subuno)\bigr)
\end{split}
\]
for all $x\subo\in A\subo$, $y\subo\in B\subo$ and $x\subuno\in A\subuno$,
$y\subuno\in B\subuno$. 

The homomorphism $m$ is a morphism in the category $\sVec$.
\end{recipe}

Given any object $\cA$ in $\Repe$, take a splitting $\cA=\cA_0\oplus\cA_1\oplus\cA_2$ as in \eqref{eq:A0A1A2}, and a refinement 
$\cA=A\subo\oplus A\subuno\oplus\delta(A\subuno)\oplus\cA_2$ as in \eqref{eq:C0C1C1C2}. Consider the object $A=A\subo\oplus A\subuno$ in $\sVec$, and the linear map $\iota_A:F(A)\rightarrow \cA$ defined as follows:
\begin{equation}\label{eq:iotaA}
\iota_A(x\subo)=x\subo,\quad \iota_A(x\subuno\otimes v_0)=x\subuno,
\quad\iota_A(x\subuno\otimes v_1)=\delta(x\subuno).
\end{equation}
This is a homomorphism in $\Repe$, that takes $F(A)$ isomorphically to 
$\cA'=\cA_0\oplus\cA_1$ and, as $\cA_2$ is isomorphic to $0$ in $\Ver$, $[\iota_A]$ turns out to be an isomorphism in $\Ver$.

\begin{theorem}\label{th:main}
Let $\mu:\cA\otimes\cB\rightarrow\cC$ be a homomorphism in $\Repe$. Pick
splittings of $\cA$, $\cB$, and $\cC$ as in  \eqref{eq:A0A1A2}, and  refinements  as in \eqref{eq:C0C1C1C2}. Define a homomorphism $m:A\otimes B\rightarrow C$
in $\sVec$ by means of Recipe \ref{recipe1}. Then, with $\iota_A,\iota_B,\iota_C$ 
as in 
\eqref{eq:iotaA}, the diagram \eqref{eq:FABCmuFm} is commutative.
\end{theorem}

\begin{proof}
Because of Lemma \ref{le:mu'}, it is enough to prove that the diagram (in $\Ver$)
\begin{equation}\label{eq:FABCmuFm'}
\begin{tikzcd}
F(A)\otimes F(B)\arrow[r, "J_{A,B}"]\arrow[d, "\hbox{$[\iota_A\otimes\iota_B]$}"']&
F(A\otimes B)\arrow[r, "F(m)"]&F(C)\arrow[d, "\hbox{$[\iota_C]$}"]\\
\cA'\otimes\cB'\arrow[rr, "\hbox{$[\mu']$}"]&&\cC'
\end{tikzcd}
\end{equation}
is commutative. (Here we use the same notation $\iota_A$ to denote the isomorphism
$F(A)\simeq \cA'$ in $\Repe$ induced by the original $\iota_A:F(A)\rightarrow\cA$ in
\eqref{eq:iotaA}.)

Using the inverse of $J_{A,B}$ (see \eqref{eq:J'}), this is equivalent to checking that in the next diagram in $\Repe$:
\[
\begin{tikzcd}
F(A)\otimes F(B)\arrow[d, "\iota_A\otimes\iota_B"']&
F(A\otimes B)\arrow[l, "j'_{A,B}"']\arrow[rr, 
"m\subo\oplus(m\subuno\otimes\id_{V_1}\hspace*{-2pt})"]&&F(C)\arrow[d, "\iota_C"]\\
\cA'\otimes\cB'\arrow[rrr, "\mu'"]&&&\cC'
\end{tikzcd}
\]
the difference $\Phi\bydef 
\iota_C\circ\bigl(m\subo\oplus(m\subuno\otimes\id_{V_1})\bigr)
- \mu'\circ(\iota_A\otimes\iota_B)\circ j'_{A,B}$ is negligible.

For $x\subo\in A\subo,y\subo\in B\subo$ we get
\[
\begin{split}
&x\subo\otimes y\subo\xrightarrow{j'_{A,B}}x\subo\otimes y\subo
\xrightarrow{\iota_A\otimes\iota_B} x\subo\otimes y\subo
\xrightarrow{\mu'}\mu'(x\subo\otimes y\subo),\\
&x\subo\otimes y\subo\xrightarrow{m\subo}m(x\subo\otimes y\subo)
 =\mu'(x\subo\otimes y\subo)
\xrightarrow{\iota_C}\mu'(x\subo\otimes y\subo),
\end{split}
\]
so that $\Phi$ is trivial on $A\subo\otimes B\subo$. In the same vein, for 
$x\subuno\in A\subuno,y\subuno\in B\subuno$ we get
\[
\begin{split}
&x\subuno\otimes y\subuno\xrightarrow{j'_{A,B}}
\frac{1}{2}\bigl((x\subuno\otimes v_0)\otimes (y\subuno\otimes v_1)-
       (x\subuno\otimes v_1)\otimes (y\subuno\otimes v_0)\bigr)\\
&\qquad\qquad \xrightarrow{\iota_A\otimes \iota_B}
  \frac{1}{2}\bigl(x\subuno\otimes\delta(y\subuno)
            -\delta(x\subuno)\otimes y\subuno\bigr)\\
&\qquad\qquad\qquad \xrightarrow{\mu'} 
      \frac{1}{2}\Bigl(\mu'\bigl(x\subuno\otimes\delta(y\subuno)\bigr)
            -\mu'\bigl(\delta(x\subuno)\otimes y\subuno\bigr)\Bigr)=
            \mu'\bigl(x\subuno\otimes\delta(y\subuno)\bigr)
            \quad \text{by \eqref{eq:mu'xdeltay}}\\[8pt]
&x\subuno\otimes y\subuno\xrightarrow{m\subo}m(x\subuno\otimes y\subuno)
   =\mu'\bigl(x\subuno\otimes \delta(y\subuno)\bigr)
\xrightarrow{\iota_C}\mu'\bigl(x\subuno\otimes \delta(y\subuno)\bigr),
\end{split}
\]
and $\Phi$ is trivial too on $A\subuno\otimes B\subuno$. Now, for $x\subo\in A\subo$
and $y\subuno\in B\subuno$, and for $\alpha,\beta\in \FF$, we get
\[
\begin{split}
&(x\subo\otimes y\subuno)\otimes(\alpha v_0+\beta v_1)
 \xrightarrow{j'_{A,B}} x\subo\otimes 
              \bigl(y\subuno\otimes (\alpha v_0+\beta v_1)\bigr)\\
&\qquad\qquad\qquad\qquad \xrightarrow{\iota_A\otimes\iota_B} x\subo\otimes
                      \bigl(\alpha y\subuno+\beta \delta(y\subuno)\bigr)
 \xrightarrow{\mu'} \mu'\bigl(\alpha x\subo\otimes y\subuno 
      +\beta x\subo\otimes\delta(y\subuno)\bigr),
 \\[10pt]
&(x\subo\otimes y\subuno)\otimes(\alpha v_0+\beta v_1)
 \xrightarrow{m\subuno\otimes\id_{V_1}} m(x\subo\otimes y\subuno)\otimes 
          (\alpha v_0+\beta v_1)\\
&\qquad\qquad\qquad\qquad\qquad\qquad\qquad 
   \xrightarrow{\ \iota_C\ } \alpha m(x\subo\otimes y\subuno) 
                             +\beta\delta\bigl(m(x\subo\otimes y\subuno)\bigr).
\end{split}
\]
But $\mu'(x\subo\otimes y\subuno)=a\subuno+\delta(b\subuno)$ for some 
$a\subuno,b\subuno\in C\subuno$. This gives $m(x\subo\otimes y\subuno)=a\subuno$. Also, as $x\subo$ is fixed by $\sigma$,
$\delta(a\subuno)=\delta\bigl(\mu'(x\subo\otimes y\subuno)\bigr)
  =\mu'\bigl(x\subo\otimes \delta(y\subuno)\bigr)$. As a consequence, we get 
\begin{multline*}
\Phi\bigl((x\subo\otimes y\subuno)\otimes (\alpha v_0+\beta v_1)\bigr)\\
=\alpha \bigl(\mu'(x\subo\otimes y\subuno) -m(x\subo\otimes y\subuno)\bigr)
+\beta \Bigl(\mu'\bigl(x\subo\otimes\delta(y\subuno)\bigr)-
    \delta\bigl(m(x\subo\otimes y\subuno)\bigr)\Bigr)
=\alpha\delta(b\subuno)\in \delta(C\subuno).
\end{multline*}
It follows that the restriction $\Phi\vert_{(A\subo\otimes B\subuno)\otimes V_1}$,
takes $(A\subo\otimes B\subuno)\otimes V_1$, which is a direct sum of copies of
$V_1$, to $\delta(C\subuno)$, which is a direct sum of copies of $V_0$, and hence it is negligible. In the same vein, the restriction 
$\Phi\vert_{(A\subuno\otimes B\subo)\otimes V_1}$ is negligible.

We conclude that $\Phi$ is negligible, as required.
\end{proof}

In particular, if $(\cA,\mu)$ is an algebra in $\Repe$, with $\mu(x\otimes y)=xy$ for all $x,y$, and we fix a splitting of $\cA$ as in \eqref{eq:A0A1A2} and a refinement $\cA=A\subo\oplus A\subuno\oplus \delta(A\subuno)\oplus\cA_2$ as in \eqref{eq:C0C1C1C2}, Recipe \ref{recipe1} becomes the following one.
(The reader should compare with \cite[Proposition 3.4]{Kannan}.)

\begin{recipe}\label{recipe} Take projections relative to this splitting, and define a multiplication $m$ 
($m(x\otimes y)\bydef x\bullet y$) on $A\bydef A\subo\oplus A\subuno$ as follows:
\[
\begin{split}
x\subo\bullet y\subo &=\mathrm{proj}_{A\subo} (x\subo y\subo)\\
x\subo\bullet y\subuno&=\mathrm{proj}_{A\subuno} (x\subo y\subuno)\\
x\subuno\bullet y\subo&=\mathrm{proj}_{A\subuno} (x\subuno y\subo)\\
x\subuno\bullet y\subuno&=\mathrm{proj}_{A\subo} 
                    \bigl(x\subuno\delta(y\subuno)\bigr)
\end{split}
\]
for all $x\subo,y\subo\in A\subo$ and $x\subuno,y\subuno\in A\subuno$. 

The algebra $(A,m)$ is an algebra in $\sVec$ (a superalgebra).
\end{recipe}

In this case, Theorem \ref{th:main} restricts to the following Corollary:

\begin{corollary}\label{co:main}
Let $(\cA,\mu)$ be an algebra in $\Repe$, with $\mu(x\otimes y)=xy$ for all $x,y$.
Pick a splitting $\cA=\cA_0\oplus\cA_1\oplus\cA_2$ as in \eqref{eq:A0A1A2}, and a refinement $\cA=A\subo\oplus A\subuno\oplus\delta(A\subuno)\oplus\cA_2$ as in \eqref{eq:C0C1C1C2}. Define a multiplication in $A=A\subo\oplus A\subuno$ by means of Recipe \ref{recipe}.

Then the algebras $(\cA,[\mu])$ and $\bigl(F(A),F(m)\circ J_{A,A}\bigr)$ in $\Ver$ are isomorphic.

In other words, $(A,m)$ is the superalgebra that corresponds to the `semisimplification' of $(\cA,\mu)$.
\end{corollary}

\bigskip

\section{From octonions to composition superalgebras}\label{se:octonions}

The notion of composition algebra on a symmetric tensor category over a field of characteristic not $2$ will be considered here. The order $3$ automorphisms of the Cayley algebras, i.e., of the eight-dimensional unital composition algebras, were determined in \cite{Canadian}. In particular,  any such automorphism on a Cayley algebra over a field of characteristic $3$ allows us to view the Cayley algebra as an algebra in $\Repe$, and hence to obtain, through the semisimplification functor
in \eqref{eq:S}, an algebra in $\Ver$ and  thus, through the equivalence $F$ in \eqref{eq:F}, a composition superalgebra.

\subsection{Composition algebras in a symmetric tensor category}\label{ss:composym}

A \emph{composition algebra} over a field $\FF$ is a triple $(\cC,\mu,\boldnup)$, where
$\mu:\cC\otimes\cC\rightarrow \cC$, $\mu(x\otimes y)=xy$  is the multiplication of
$\cC$, and $\boldnup:\cC\rightarrow \FF$ is a nonsingular multiplicative quadratic form, called the \emph{norm}.
Here nonsingular means that either the polar form 
$\boldnup(x,y)\bydef \boldnup(x+y)-\boldnup(x)-\boldnup(y)$ is a nondegenerate bilinear form, or the
characteristic of $\FF$ is $2$ and there is no nonzero element such that 
$\boldnup(x,\cC)=0=\boldnup(x)$. Note that the same symbol is used to denote the norm and its polar form. Also, the polar form may be considered as a linear map
$\boldnup:\cC\otimes\cC\rightarrow\FF$.
The norm being multiplicative means that the equation $\boldnup(xy)=\boldnup(x)\boldnup(y)$ holds for all $x,y\in\cC$.

Unital composition algebras (also termed Hurwitz algebras) over a field are the 
analogues of the classical algebras or real and complex numbers, 
quaternions, and octonions. In particular their dimension is restricted to $1$, $2$, $4$ or $8$. The reader may consult \cite[Chapter 2]{ZSSS}, \cite[Chapter VIII]{KMRT}, or the survey paper \cite{Eld_comp}.

Assume in the rest of the section that the characteristic of the ground field $\FF$ is not $2$. 

Linearizing twice the multiplicative identity one gets
\begin{equation}\label{eq:multiplicative}
\boldnup(xy,zt)+\boldnup(zy,xt)=\boldnup(x,z)\boldnup(y,t)
\end{equation}
for all $x,y,z,t\in\cC$, and conversely, the characteristic being not $2$, \eqref{eq:multiplicative} gives, with $z=x$ and $t=y$, the multiplicative condition $\boldnup(xy)=\boldnup(x)\boldnup(y)$.

Now, we may define a \emph{composition algebra} in a symmetric tensor category 
$\mathfrak{C}$ as an object $\cA$ endowed with  morphisms $\mu:\cA\otimes\cA\rightarrow \cA$ and $\boldnup:\cA\otimes\cA\rightarrow \mathbf{1}$, such that the following conditions are satisfied:
\begin{description}
\item[Symmetry]
$\boldnup\circ c_{\cA,\cA}=\boldnup$, where $c_{\cA,\cA}\in \End_{\mathfrak{C}}(\cA\otimes\cA)$ is the symmetric braiding.

\item[Multiplicativity]
The following equality of morphisms $\cA^{\otimes 4}\rightarrow \mathbf{1}$, generalizing \eqref{eq:multiplicative}, holds:
\[
\boldnup\circ(\mu\otimes\mu)\circ(\id + c_{13})=(\boldnup\otimes\boldnup)\circ c_{23}
\]
where we omit the isomorphism $\mathbf{1}\otimes\mathbf{1}\simeq\mathbf{1}$, and where $c_{12}=c_{\cA,\cA}\otimes \id_\cA\otimes \id_\cA$, 
$c_{23}=\id_\cA\otimes c_{\cA,\cA}\otimes \id_\cA$, and $c_{13}=c_{23}\circ c_{12}\circ c_{23}$.

\item[Nondegeneracy]
The composition
\[
\cA\xrightarrow{\id_\cA\otimes\mathrm{coev}_\cA}\cA\otimes\cA\otimes\cA^*
\xrightarrow{\boldnup\otimes \id_{\cA^*}}\cA^*
\]
is an isomorphism.  (We omit the associative and unitor morphisms, and 
$\mathrm{coev}_\cA$ denotes the coevaluation morphism 
$\mathbf{1}\rightarrow \cA\otimes \cA^*$ in the symmetric tensor category $\mathfrak{C}$.)
\end{description}

\smallskip

Assume now that the characteristic of the ground field $\FF$ is $3$, and let 
$(\cA,\mu,\boldnup)$ be a composition algebra endowed with an automorphism $\sigma$ with 
$\sigma^3=\id$. (This means that $\sigma$ leaves invariant both $\mu$ and 
$\boldnup$.) Then, looking at the polar form as a linear map 
 $\boldnup:\cA\otimes\cA\rightarrow \FF$, the triple $(\cA,\mu,\boldnup)$ is a composition algebra in $\Repe$.

\begin{lemma}\label{le:norm}
Let $(\cA,\mu,\boldnup)$ be a composition algebra endowed with an automorphism $\sigma$ with $\sigma^3=\id$. Let $\cA=\cA_0\oplus\cA_1\oplus\cA_2$ be a splitting as in
\eqref{eq:A0A1A2}. Then, with $\delta=\sigma-\id$, the following conditions hold:
\begin{enumerate}
\item $\boldnup\bigl(\ker\delta,\delta(\cA)\bigr)=0$,
\item $\boldnup\bigl(\delta(\cA_1),\delta(\cA)\bigr)=0$,
\item $\boldnup(\delta(x),y)+\boldnup(x,\delta(y))=0$ for all $x\in\cA_1$ and $y\in \cA$.
\end{enumerate}
\end{lemma}
\begin{proof}
For any $x,y\in \cA$, $\boldnup(x,y)=\boldnup\bigl(\sigma(x),\sigma(y)\bigr)=\boldnup\bigl(x+\delta(x),y+\delta(y)\bigr)$, and this gives 
\begin{equation}\label{eq:normdelta}
\boldnup\bigr(\delta(x),y\bigr)+\boldnup\bigl(x,\delta(y)\bigr)
+\boldnup\bigl(\delta(x),\delta(y)\bigr)=0.
\end{equation}
If $\delta(x)=0$, then $\boldnup\bigl(x,\delta(y)\bigr)=0$ for all $y$, proving the first assertion. The second part follows since $\delta(\cA_1)$ is contained in 
$\ker\delta$, and hence \eqref{eq:normdelta} gives the third assertion. 
\end{proof}

Apply the semisimplification functor $S$ in \eqref{eq:S} to get
a composition algebra $(\cA,[\mu],[\boldnup])$ in $\Ver$.

As $\boldnup:\cA\otimes \cA\rightarrow \FF$ is a morphism in $\Repe$ ($\FF$ being a 
trivial object in $\Repe$: $\FF=\FF_0$), Lemma \ref{le:mu'} becomes, in this case, the next result:

\begin{lemma}\label{mu'n'}
Let $(\cA,\mu,\boldnup)$ be a composition algebra in $\Repe$.
Then the\linebreak composition algebra $(\cA,[\mu],[\boldnup])$ in $\Ver$ is isomorphic to the composition algebra $(\cA',[\mu'],[\boldnup'])$, where $\cA'=\cA_0\oplus \cA_1$, $\mu'\in\Hom_{\Repe}(\cA'\otimes\cA',\cA')$ as in  Corollary \ref{co:mu'} and $\boldnup'$ is given by the formula
\[
\boldnup'(x\otimes y)\bydef\begin{cases}
\boldnup(x\otimes y)&\text{for $x,y\in\cA_0$ or $x,y\in\cA_1$,}\\
0&\text{for $x\in\cA_0,y\in\cA_1$, or $x\in\cA_1,y\in\cA_0$.}
\end{cases}
\]
\end{lemma}

Recipe \ref{recipe1} with $\cA=\cB$ and $\cC=\FF$ gives the following: 

\begin{recipe}\label{recipe2}
Let $(\cA,\mu,\boldnup)$ be a composition algebra in $\Repe$. Take $A\subo=\cA_0$ and 
$A\subuno$ as in \eqref{eq:C0C1C1C2}, and  define a bilinear map 
$\nup$ 
on $A=A\subo\oplus A\subuno$ (or equivalently a linear map 
$A\otimes A\rightarrow \FF$) as follows:
\[
\begin{split}
\nup(x\subo,y\subo)&= \boldnup(x\subo,y\subo)\\
\nup(x\subo,y\subuno)&=0=\nup(y\subuno,x\subo)\\
\nup(x\subuno,y\subuno)&=\boldnup\bigl(x\subuno,\delta(y\subuno)\bigr)
\end{split}
\]
for all $x\subo,y\subo\in A\subo$ and $x\subuno,y\subuno\in A\subuno$.

Note that Lemma \ref{le:norm} gives 
$\nup(x\subuno,y\subuno)=-\nup(y\subuno,x\subuno)$, so $\nup$ is `supersymmetric'.
\end{recipe}

And finally, Theorem \ref{th:main} gives our next result:

\begin{theorem}\label{th:main2}
Let $(\cA,\mu,\boldnup)$ be a composition algebra in $\Repe$, with $\mu(x\otimes y)=xy$ for all $x,y$.
Pick a splitting $\cA=\cA_0\oplus\cA_1\oplus\cA_2$ as in \eqref{eq:A0A1A2}, and a refinement $\cA=A\subo\oplus A\subuno\oplus\delta(A\subuno)\oplus\cA_2$ as in \eqref{eq:C0C1C1C2}. Define a multiplication $m$ in $A=A\subo\oplus A\subuno$ by means of Recipe \ref{recipe}, and a norm $\nup$ as in Recipe \ref{recipe2}.

Then the composition algebras $(\cA,[\mu],[\boldnup])$ and 
$\bigl(F(A),J_{A,A}\circ F(m),J_{A,A}\circ F(\nup)\bigr)$ in $\Ver$ are isomorphic.

In other words,  $(A,m,\nup)$ is the composition superalgebra that corresponds to the `semisimplification' of $(\cA,\mu,\boldnup)$.
\end{theorem}

\bigskip

\subsection{Order \texorpdfstring{$3$}{3} automorphisms of Cayley algebras}\label{ss:order3}

A unital composition algebra (or Hurwitz algebra) of dimension $\geq 2$ is said to be \emph{split} if its norm is isotropic. For each dimension $2$, $4$ or $8$, there is a unique split Hurwitz algebra, up to isomorphism. The split Cayley algebra has a canonical basis with multiplication given in Table \ref{ta:good_basis}. The elements of the canonical basis are all isotropic and they form a hyperbolic basis:
\[
\boldnup(e_1,e_2)=1=\boldnup(u_i,v_i),\ i=1,2,3.
\]
All the other values of the polar form for basic elements are either $0$ or follow from the above by using that $\boldnup$ is symmetric. Note that the $u_i$'s generate the whole algebra.

\begin{table}[!h]
\[
\vcenter{\offinterlineskip
\halign{\hfil$#$\enspace\hfil&#\vreglon
 &\hfil\enspace$#$\enspace\hfil
 &\hfil\enspace$#$\enspace\hfil&#\vregleta
 &\hfil\enspace$#$\enspace\hfil
 &\hfil\enspace$#$\enspace\hfil
 &\hfil\enspace$#$\enspace\hfil&#\vregleta
 &\hfil\enspace$#$\enspace\hfil
 &\hfil\enspace$#$\enspace\hfil
 &\hfil\enspace$#$\enspace\hfil&#\vreglon\cr
 &\omit\hfil\vrule width 1pt depth 4pt height 10pt
   &e_1&e_2&\omit&u_1&u_2&u_3&\omit&v_1&v_2&v_3&\cr
 \noalign{\hreglon}
 e_1&& e_1&0&&u_1&u_2&u_3&&0&0&0&\cr
 e_2&&0&e_2&&0&0&0&&v_1&v_2&v_3&\cr
 &\multispan{12}{\hregletafill}\cr
 u_1&&0&u_1&&0&v_3&-v_2&&-e_1&0&0&\cr
 u_2&&0&u_2&&-v_3&0&v_1&&0&-e_1&0&\cr
 u_3&&0&u_3&&v_2&-v_1&0&&0&0&-e_1&\cr
 &\multispan{12}{\hregletafill}\cr
 v_1&&v_1&0&&-e_2&0&0&&0&u_3&-u_2&\cr
 v_2&&v_2&0&&0&-e_2&0&&-u_3&0&u_1&\cr
 v_3&&v_3&0&&0&0&-e_2&&u_2&-u_1&0&\cr
 \noalign{\hreglon}}}
\]
\caption{{\vrule width 0pt height 15pt}
Multiplication table of the split Cayley algebra.}
\label{ta:good_basis}
\end{table}

The subalgebra spanned by the orthogonal idempotents $e_1$ and $e_2$ is the split
Hurwitz algebra in dimension $2$, while the subalgebra spanned by $e_1,e_2,u_1,v_1$ is the split quaternion algebra.

Among Cayley algebras (i.e., eight-dimensional Hurwitz algebras) over a field 
$\FF$ of characteristic $3$, only the split one is endowed with order $3$ automorphisms. The order $3$ automorphisms are then classified, up to conjugacy,
in this  theorem.

\begin{theorem}[\hbox{\cite[Theorem 6.3]{Canadian}}]\label{th:order3}
Let $(\cC,\mu,\boldnup)$ be a Cayley algebra over a field $\FF$ of characteristic $3$, and let $\sigma$ be an order $3$ automorphism of $(\cC,\mu,\boldnup)$. Then 
$(\cC,\mu,\boldnup)$ is the split Cayley algebra and one of the following conditions holds:
\begin{enumerate}
\item $(\sigma-\id)^2=0$, and there exists a canonical basis of $\cC$ such that
\[
\sigma(u_i)=u_i,\ i=1,2,\quad \sigma(u_3)=u_3+u_2.
\]

\item $(\sigma-\id)^2\ne 0$ and there is a quadratic \'etale subalgebra $\cK$ of 
$\cC$ fixed elementwise by $\sigma$. 

If $\FF$ is algebraically closed, then there is a canonical basis of $\cC$ such that 
\[
\sigma(u_i)=u_{i+1}\quad \text{(indices modulo $3$).}
\]

\item There is a canonical basis such that
\[
\sigma(u_i)=u_i,\ i=1,2,\quad \sigma(u_3)=u_3+v_3-(e_1-e_2).
\]

\item There is a canonical basis such that
\[
\sigma(u_i)=u_i,\ i=1,2,\quad \sigma(u_3)=u_3+u_2+v_3-(e_1-e_2).
\]
\end{enumerate}
\end{theorem}

It must be remarked that the automorphism in item (1) above corresponds to the 
so called \emph{quaternionic idempotents} of Okubo algebras, while the 
automorphism in item (4) corresponds to the \emph{singular idempotents} of 
Okubo algebras. These are specific  to characteristic $3$ and have no counterpart 
in other characteristics. For details, the reader may consult \cite{Canadian}. 

\bigskip

\subsection{`Semisimplification' of Cayley algebras}\label{ss:semiCayley}

Assume in this section that the characteristic of the ground field $\FF$ is $3$.

For each of the possibilities in Theorem \ref{th:order3}, the unital composition superalgebra that corresponds to the semisimplification of the Hurwitz algebra 
$(\cC,\mu,\boldnup)$ will be determined here. In order to do this, it is enough to apply Recipes \ref{recipe} and \ref{recipe2}.

\begin{enumerate}
\item 
$\sigma(u_i)=u_i,\ i=1,2,\quad \sigma(u_3)=u_3+u_2$. Then $e_1$, $e_2$, $v_1$, and $v_3$ are fixed by $\sigma$, while $\sigma(v_2)=v_2-v_3$. With $\delta=\sigma-\id$, we have $u_3\xrightarrow{\delta} u_2\xrightarrow{\delta}0$, $v_2\xrightarrow{\delta}-v_3\xrightarrow{\delta}0$, so we get a splitting
$\cC=C\subo\oplus C\subuno\oplus \delta(C\subuno)\oplus\cC_2$ as in \eqref{eq:C0C1C1C2} with $C\subo=\espan{e_1,e_2,u_1,v_1}$, $C\subuno=\espan{u_3,v_2}$ and $\cC_2=0$. The multiplication in 
$C=C\subo\oplus C\subuno$ is given by the table:
\begin{equation}\label{eq:B42}
\vcenter{\offinterlineskip
\halign{\hfil$#$\enspace\hfil&#\vreglon
 &\hfil\enspace$#$\enspace\hfil
 &\hfil\enspace$#$\enspace\hfil
 &\hfil\enspace$#$\enspace\hfil
 &\hfil\enspace$#$\enspace\hfil&#\vregleta
 &\hfil\enspace$#$\enspace\hfil
 &\hfil\enspace$#$\enspace\hfil&#\vreglon\cr
 &\omit\hfil\vrule width 1pt depth 4pt height 10pt
   &e_1&e_2&u_1&v_1&\omit&u_3&v_2&\cr
 \noalign{\hreglon}
 e_1&& e_1&0&u_1&0&&u_3&0&\cr
 e_2&& 0&e_2&0&v_1&&0&v_2&\cr
 u_1&&0&u_1&0&-e_1&&-v_2&0&\cr
 v_1&&v_1&0&-e_2&0&&0&u_3&\cr
 &\multispan{9}{\hregletafill}\cr
 u_3&&0&u_3&v_2&0&&-v_1&e_1&\cr
 v_2&&v_2&0&0&-u_3&&-e_2&-u_1&\cr
 \noalign{\hreglon}}}
\end{equation}
The norm $\nup$ restricts to $\boldnup$ on the even part $C\subo$, and satisfies 
$\nup(u_3,v_2)=\boldnup\bigl(u_3,\delta(v_2)\bigr)
=\boldnup(u_3,-v_3)=-1=-\nup(v_2,u_3)$. 

This composition superalgebra is the superalgebra $B(4,2)$ in \cite{Shestakov,EldOkubo}.

\smallskip

\item There is a quadratic \'etale subalgebra $\cK$ of 
$\cC$ fixed elementwise by $\sigma$, and the action of $\sigma$ on $\cK^\perp$ (orthogonal relative to $\boldnup$) is given by two cycles of length $3$. This gives the decomposition in \eqref{eq:A0A1A2}
with $\cC_0=\cK$, $\cC_1=0$ and $\cC_2=\cK^\perp$. Then the semisimplification simply gives the composition algebra $\cK$ with trivial odd component.

\smallskip

\item $u_1$ and $u_2$ are fixed by $\sigma$, while $\sigma(u_3)=u_3+v_3-(e_1-e_2)$.
We get the following chains (Jordan blocks) for the action of $\delta=\sigma-\id$:
\[
\begin{split}
&u_3\xrightarrow{\ \delta\ } v_3-(e_1-e_2)
             \xrightarrow{\ \delta\ }-v_3\xrightarrow{\delta}0,\\
&v_1\xrightarrow{\ \delta\ } u_2\xrightarrow{\ \delta\ } 0,\\
&v_2\xrightarrow{\ \delta\ } -u_1\xrightarrow{\ \delta\ } 0,\\
&1\xrightarrow{\ \delta\ }0,
\end{split}
\]
so we get a splitting
$\cC=C\subo\oplus C\subuno\oplus \delta(C\subuno)\oplus\cC_2$ as in \eqref{eq:C0C1C1C2} with $C\subo=\FF 1$, $C\subuno=\espan{v_1,v_2}$ and 
$\cC_2=\espan{u_3,v_3,e_1-e_2}$. The multiplication in 
$C=C\subo\oplus C\subuno$ is given by the table:
\begin{equation}\label{eq:B12}
\vcenter{\offinterlineskip
\halign{\hfil$#$\enspace\hfil&#\vreglon
 &\hfil\enspace$#$\enspace\hfil&#\vregleta
 &\hfil\enspace$#$\enspace\hfil
 &\hfil\enspace$#$\enspace\hfil&#\vreglon\cr
 &\omit\hfil\vrule width 1pt depth 4pt height 10pt
   &1&\omit&v_1&v_2&\cr
 \noalign{\hreglon}
 1&&1&& v_1&v_2&\cr
 &\multispan{6}{\hregletafill}\cr
 v_1&&v_1&&0&-1&\cr
 v_2&&v_2&&1&0&\cr
 \noalign{\hreglon}}}
\end{equation}
The norm satisfies $\nup(1,1)=2=-1$, and 
$\nup(v_1,v_2)=\boldnup\bigl(v_1,\delta(v_2)\bigr)=\boldnup(v_1,-u_1)=-1
=-\nup(v_2,v_1)$.

This composition superalgebra is the superalgebra $B(1,2)$ in \cite{Shestakov,EldOkubo}.

\item $u_1$ and $u_2$ are fixed by $\sigma$, while $\sigma(u_3)=u_3+u_2+v_3-(e_1-e_2)$.
We get the following chains for the action of $\delta=\sigma-\id$:
\[
\begin{split}
&u_3\xrightarrow{\ \delta\ } u_2+v_3-(e_1-e_2)
             \xrightarrow{\ \delta\ }-v_3\xrightarrow{\ \delta\ }0,\\
&v_1\xrightarrow{\ \delta\ } u_2\xrightarrow{\ \delta\ } 0,\\
&v_2\xrightarrow{\ \delta\ } -v_3-u_1\xrightarrow{\ \delta\ } 0,\\
&1\xrightarrow{\ \delta\ }0,
\end{split}
\]
so we get a splitting
$\cC=C\subo\oplus C\subuno\oplus \delta(C\subuno)\oplus\cC_2$ as in \eqref{eq:C0C1C1C2} with $C\subo=\FF 1$, $C\subuno=\espan{v_1,v_2}$ and 
$\cC_2=\espan{u_3,v_3,u_2-(e_1-e_2)}$. The multiplication table on $C=C\subo\oplus C\subuno$ and its norm coincide with those in the previous case.

\end{enumerate}

\bigskip

\section{Semisimplification: skew transformations, derivations}\label{se:ss}

This last section will show some features of the semisimplification process.
The Lie algebra of derivations of an algebra $(\cA,\mu)$ in $\Repe$ is also an algebra in $\Repe$ in a natural way, but its semisimplification may fail to be the Lie superalgebra of derivations of the semisimplification of $(\cA,\mu)$. However,
the semisimplification of the Lie algebra of the skew-symmetric transformations, relative to the norm, of a composition algebra in $\Repe$ is isomorphic to the 
orthosymplectic Lie superalgebra of skew-transformations (in the super setting) of the corresponding composition superalgebra.

Throughout the section, the characteristic of the ground field $\FF$ will be assumed to be $3$.

\subsection{Skew transformations}\label{ss:skew}

Given an object $\cV$ in $\Repe$, we will denote by $\cV^{ss}$ an object in $\sVec$,
 such that $F(\cV^{ss})$ and $\cV$  ($=S(\cV)$ for $S$ in \eqref{eq:S}) are isomorphic as objects in $\Ver$. The vector superspace
 $\cV^{ss}$ will be called a \emph{semisimplification} of $\cV$. In the same vein, given an algebra $(\cA,\mu)$ in $\Repe$, we will denote by $(\cA^{ss},\mu^{ss})$ a superalgebra (i.e., an algebra in $\sVec$) such that 
$\bigl(F(\cA^{ss}),F(\mu^{ss})\circ J_{\cA^{ss},\cA^{ss}}\bigr)$ is isomorphic to the algebra $(\cA,[\mu])$ in $\Ver$. The multiplication $\mu$ will be omitted if it is clear from the context,

The same applies to vector spaces endowed with a bilinear form: 
$(\cV^{ss},\boldbup^{ss})$; or to composition algebras 
$(\cC^{ss},\mu^{ss},\boldnup^{ss})$.

\begin{proposition}\label{pr:endos_skew}
Let $\cV$ be an object in $\Repe$. 
\begin{itemize}
\item 
The associative superalgebras  $\End_\FF(\cV^{ss})$ and 
$\bigl(\End_\FF(\cV)\bigr)^{ss}$ are isomorphic.
\item Let $\boldbup:\cV\otimes\cV\rightarrow \FF$ be a morphism in $\Repe$ such that the bilinear form (denoted by the same symbol) given by 
$\boldbup(x,y)\bydef \boldbup(x\otimes y)$ is  symmetric and  nondegenerate. Then the bilinear form corresponding to the morphism 
$\boldbup^{ss}:\cV^{ss}\otimes \cV^{ss}\rightarrow \FF$ in $\sVec$ is 
super-symmetric and nondegenerate, and the orthosymplectic Lie algebra 
$\frosp(\cV^{ss},\boldbup^{ss})$ is isomorphic to the semisimplification 
$\frso(\cV,\boldbup)^{ss}$.
\end{itemize}
\end{proposition}
\begin{proof}
Note first that  $\End_\FF(\cV)$ is isomorphic to $\cV\otimes \cV^{*}$ as objects in $\Repe$, 
where the element $v\otimes f$ corresponds to the endomorphism $w\mapsto vf(w)$, for $v,w\in\cV$ and $f\in\cV^*$. The multiplication in $\cV\otimes\cV^*$ is given by 
the following composition (associative and unitor morphisms are omitted, as usual) 
involving the evaluation morphism 
$\mathrm{ev}_{\cV}:\cV^*\otimes\cV\rightarrow \FF$:
\[
\cV\otimes\cV^*\otimes\cV\otimes\cV^*\xrightarrow{\id_\cV\otimes 
\mathrm{ev}_{\cV}\otimes\id_{\cV^*}}\cV\otimes\cV^*.
\]
The first part follows at once because the semisimplification functor $S$ in \eqref{eq:S} is a braided  monoidal functor (see \cite[Definition 8.1.7]{EGNO}), and the equivalence $F$ in \eqref{eq:F} is a braided  monoidal equivalence.

For the second part, the symmetry of $\boldbup^{ss}$ in $\sVec$ (that is, the fact that $\boldbup^{ss}$ is super-symmetric) and its nondegeneracy are again consequences of the fact that $S$ and $F$ are braided  monoidal functors. In this case, the algebra 
 $\End_\FF(\cV)$ is isomorphic to $\cV\otimes\cV$, where the element $v\otimes w$ corresponds to the linear map $x\mapsto v\boldbup(w,x)$, and the multiplication is given by the composition
\[
\cV\otimes\cV\otimes\cV\otimes\cV\xrightarrow{\id_\cV\otimes \boldbup\otimes\id_{\cV}}\cV\otimes\cV.
\]
The corresponding orthogonal Lie algebra $\frso(\cV,\boldbup)$ corresponds to the subspace $\Skew^2(\cV\otimes\cV)$ of skew-symmetric tensors, which is the image of the projection 
$P=\frac{1}{2}(\id_{\cV\otimes\cV}-c_{\cV,\cV})\in\End_{\Repe}(\cV\otimes\cV)$. As usual, 
$c_{\cV,\cV}:\cV\otimes\cV\rightarrow \cV\otimes\cV$ is the braiding (the usual swap in this case).

Since $S$ and $F$ are braided  monoidal functors, the semisimplification 
 $\frso(\cV,\boldbup)^{ss}$ is isomorphic to the image of the projection 
$\frac{1}{2}(\id_{\cV^{ss}\otimes\cV^{ss}}-c^{ss}_{\cV^{ss},\cV^{ss}})$, where now the braiding $c^{ss}_{\cV^{ss},\cV^{ss}}$ is given by the parity swap. This is the subspace of super-skew-symmetric tensors in $\cV^{ss}\otimes\cV^{ss}$, and this, in turn, is isomorphic to the orthosymplectic Lie superalgebra 
$\frosp(\cV^{ss},\boldbup^{ss})$.
\end{proof}

\smallskip

\subsection{Derivations}\label{ss:derivations}

As mentioned at the beginning of the section, derivations present a different behavior under semisimplification. Note that any automorphism $\tau$ of an algebra 
$(\cA,\mu)$ induces an automorphism $\Ad_\tau: d\mapsto \tau\circ d\circ \tau^{-1}$ in its Lie algebra of derivations.

To begin with, given a Lie algebra 
$(\cL,\mu_\cL)$,  an algebra $(\cA,\mu_\cA)$ in 
$\Repe$, and a morphism  $\Phi:\cL\otimes\cA\rightarrow\cA$ in $\Repe$ given by 
$x\otimes a\mapsto x.a$; $\Phi$ is an action by derivations of $\cL$ on $\cA$ if and only if
the following two conditions are satisfied for all $x,y\in\cL$ and $a,b\in\cA$:
\[
[x,y].a=x.(y.a)-y.(x.a),\quad
x.(ab)=(x.a)b+a(x.b),
\]
where $[x,y]=\mu_\cL(x\otimes y)$ and $ab=\mu_{\cA}(a\otimes b)$. This can be written as follows:
\[
\begin{split}
&\Phi\circ(\mu_{\cL}\otimes\id_{\cA})
=\Phi\circ(\id_{\cL}\otimes\Phi)\circ 
    \bigl(\id_{\cL\otimes\cL\otimes\cA} -c_{\cL,\cL}\otimes\id_\cA\bigr)\\
&\Phi\circ(\id_\cL\otimes\mu_{\cA})
=\mu_{\cA}\circ\bigl(\Phi\otimes\id_{\cA}+
    (\id_{\cA}\otimes\Phi)\circ(c_{\cL,\cA}\otimes\id_{\cA})\bigr)
\end{split}
\]
where the first equality holds in 
$\Hom_{\Repe}(\cL\otimes\cL\otimes\cA,\cA)$, while the second holds in
$\Hom_{\Repe}(\cL\otimes\cA\otimes\cA,\cA)$,
and therefore all this goes smoothly under semisimplification.

As a consequence, we obtain our next result:

\begin{proposition}\label{pr:der}
For any algebra $(\cA,\mu)$ in $\Repe$, there is a natural homomorphism 
$\Der(\cA,\mu)^{ss}\rightarrow \Der(\cA^{ss},\mu^{ss})$ from the semisimplification of the Lie algebra of derivations of $(\cA,\mu)$ into the Lie superalgebra of derivations of the superalgebra $(\cA^{ss},\mu^{ss})$.
\end{proposition}

We will compute next the semisimplification of the Lie algebras of derivations of the algebras $(\cC,\mu,\boldnup)$ in cases (1), (3) and (4) of subsection \ref{ss:semiCayley}. As in subsection \ref{ss:semiCayley}, the situation in case (2)
is quite trivial.

Take the canonical basis of the split Cayley algebra $\cC$ as in Table \ref{ta:good_basis}, and write $\cU=\FF u_1+\FF u_2+\FF u_3$, 
$\cV=\FF v_1+\FF v_2+\FF v_3$. The characteristic of $\FF$ being $3$ implies that the Lie algebra of derivations $\Der(\cC)$ splits as
(see \cite[Proposition 4.29]{EKmon})
\begin{equation}\label{eq:SUV}
\Der(\cC)=\cS\oplus \ad_\cU\oplus \ad_\cV,
\end{equation}
where, as usual, $\ad_x(y)=[x,y]$, and where $\cS=\{d\in\Der(\cC)\mid d(e_1)=0=d(e_2)\}$. 

At this point, it should be remarked that, the characteristic being $3$, 
$\Der(\cC)$ is not the contragredient
Lie algebra attached to the Cartan matrix 
$\left(\begin{smallmatrix} 2&-3\\ -1&2\end{smallmatrix}\right)$. (See 
\cite[Example 3.4]{Kannan}.) This contragredient Lie algebra is, in fact, a subalgebra of $\Der(\cC)$ given by $\cS'\oplus\ad_\cU\oplus \ad_\cV$, with $\cS'$ a subalgebra of $\cS$ isomorphic to $\frgl_2(\FF)$. 

Moreover, the restriction of $\cS$ to $\cU$ gives an isomorphism  $\cS\simeq\frsl(\cU)$. The subspace $\ad_\cC$ is a seven-dimensional ideal of 
$\Der(\cC)$ isomorphic to the projective special linear Lie algebra 
$\frpsl_3(\FF)$, and the quotient $\Der(\cC)/\ad_\cC$ is again isomorphic to 
$\frpsl_3(\FF)$. Under the isomorphism $\cS\simeq\frsl(\cU)$, any trace zero endomorphism $f$ of $\cU$ acts trivially on $\FF e_1+\FF e_2$, as $f$ on $\cU$, and as $-f^*$ on $\cV$, where $f^*$ is determined by the equation 
$\boldnup\bigl(f(u),v\bigr)=\boldnup\bigl(u,f^*(v)\bigr)$ for all 
$u\in \cU$ and $v\in\cV$. We will identify $\cS$ with $\frsl(\cU)$ and will denote by $E_{ij}$ the 
linear endomorphism of 
$\cU$ taking $u_j$ to $u_i$ and sending $u_{l}$ to $0$ for $l\neq j$. In particular, $\ad_{e_1-e_2}$ is identified with twice the identity map 
$I_3= E_{11}+E_{22}+E_{33}$.

\begin{itemize}
\item In case (1) of subsection \ref{ss:semiCayley}, and because $\Ad_\sigma(\ad_x)=\ad_{\sigma(x)}$ for all $x$, it is easy to compute a splitting of $\Der(\cC)$ into Jordan blocks relative to the nilpotent transformation 
$\Delta\bydef\Ad_\sigma -\id$, as follows:
\[
\begin{split}
&E_{32}\xrightarrow{\Delta} E_{22}-E_{33}-E_{23} \xrightarrow{\Delta} E_{23}                
      \xrightarrow{\Delta} 0,\\
& E_{12} \xrightarrow{\Delta} -E_{13} \xrightarrow{\Delta} 0,\\
& E_{31} \xrightarrow{\Delta} E_{21} \xrightarrow{\Delta} 0,\\
& \ad_{u_3} \xrightarrow{\Delta} \ad_{u_2} \xrightarrow{\Delta} 0,\\
& \ad_{v_2} \xrightarrow{\Delta} -\ad_{v_3} \xrightarrow{\Delta} 0,\\
& \ad_{e_1-e_2},\ad_{u_1},\ad_{v_1} \xrightarrow{\Delta} 0.
\end{split}
\]
Therefore we get a splitting $\Der(\cC)=D\subo\oplus D\subuno
\oplus \Delta(D\subuno)\oplus D_2$ as in \eqref{eq:C0C1C1C2} with 
$D\subo=\ad_{\FF(e_1-e_2)+\FF u_1+\FF v_1}$ and 
$D\subuno=\FF \ad_{u_3}\oplus \FF\ad_{v_2}\oplus \FF E_{12}\oplus\FF E_{31}$.

Recipe \ref{recipe} gives the multiplication in the Lie superalgebra 
$\Der(\cC)^{ss}=D\subo\oplus D\subuno$. The subspace 
$D\subo\oplus \FF \ad_{u_3}\oplus\FF\ad_{v_2}$ is an ideal isomorphic to 
$\frosp(1,2)$. 

Moreover, the action of $\Der(\cC)^{ss}$ on $\cC^{ss}=C\subo\oplus C\subuno$ is determined by its action on the odd part (as the odd part generates the whole superalgebra, see \eqref{eq:B42}). Using Recipe \ref{recipe1} we obtain
\[
\begin{split}
\ad_{u_3}\cdot u_3&=\ad_{u_3}(\delta(u_3))=[u_3,u_2]=v_1,\\
\ad_{u_3}\cdot v_2&=\ad_{u_3}(\delta(v_2))=[u_3,-v_3]=e_1-e_2,\\
\ad_{v_2}\cdot u_3&=\ad_{v_2}(\delta(u_3))=[v_2,u_2]=e_1-e_2,\\
\ad_{v_2}\cdot v_2&=\ad_{v_2}(\delta(v_2))=[v_2,-v_3]=u_1,\\
E_{12}\cdot u_3&=E_{12}(u_2)=u_1,\\
E_{12}\cdot v_2&=-E_{12}(v_3)=0,\\
E_{31}\cdot u_3&=0,\\
E_{31}\cdot v_2&=-E_{31}(v_3)=v_1.
\end{split}
\]
It turns out that $\Der(\cC)^{ss}=D\subo\oplus D\subuno$ is isomorphic to the Lie superalgebra $\Der(\cC^{ss})$ (see \cite[Theorem 5.8]{EldOkubo}).

\medskip

\item In case (3) of subsection \ref{ss:semiCayley}, $\cC=\cQ\oplus \cQ^\perp$, where  $\cQ$ is the quaternion subalgebra spanned by $e_1,e_2,u_3,v_3$ and 
$\cQ^\perp$ (orthogonal subspace relative to the norm) is spanned by $u_1,u_2,v_1,v_2$. This is a $\ZZ/2$-grading of $\cC$ that induces a $\ZZ/2$-grading of $\Der(\cC)$ whose even component is $\{d\in\Der(\cC)\mid d(\cQ)\subset \cQ\}$ is the span (with the notation used in the previous case) of $E_{12},E_{21},E_{11}-E_{22},\ad_{e_1-e_2},\ad_{u_3},\ad_{v_3}$. The derivations $E_{12},E_{21},E_{11}-E_{22}$ are all fixed by $\Ad_\sigma$, while the remaining three elements form a three-dimensional indecomposable module for the action of $\Ad_\sigma$ (i.e., 
isomorphic to $V_2$ in $\Repe$). The odd part
decomposes into the direct sum of the following Jordan blocks
for the linear endomorphism $\Delta=\Ad_\sigma-\id$ of $\Der(\cC)$:
\[
\begin{split}
&E_{31} \xrightarrow{\Delta} -E_{23}+\ad_{u_2}-\ad_{v_1} 
            \xrightarrow{\Delta} -\ad_{u_2}\xrightarrow{\Delta} 0,\\
&E_{32}\xrightarrow{\Delta} E_{13}-\ad_{u_1}-\ad_{u_2} 
        \xrightarrow{\Delta} \ad_{u_1} \xrightarrow{\Delta} 0,\\
& E_{13} \xrightarrow{\Delta} 0,\quad  E_{23} \xrightarrow{\Delta} 0.
\end{split}
\]
Therefore we get a splitting $\Der(\cC)=D\subo\oplus D\subuno\oplus
\Delta(D\subuno)\oplus D_2$ as in \eqref{eq:C0C1C1C2}, with $D\subuno=0$ as there are no Jordan blocks of length $2$, and
$D\subo=\espan{E_{12},E_{21},E_{11}-E_{22},E_{13},E_{23}}$. It turns out that
$\Der(\cC)^{ss}$ is a Lie algebra (its odd part is trivial) of dimension $5$, which is the direct sum of a copy of $\frsl_2(\FF)$ and a two-dimensional abelian ideal: 
$\FF E_{13}+\FF E_{23}$. This ideal is the natural two-dimensional module for the copy of $\frsl_2(\FF)$.

By \cite[proof of Lemma 5.3]{EldOkubo}, $\Der(\cC^{ss})$ is isomorphic to 
$\frsl_2(\FF)$. Actually, it turns out that the ideal $\FF E_{13}+\FF E_{23}$ in 
$\Der(\cC)^{ss}$ acts trivially on $\cC^{ss}$. (Recall that the action is given by Recipe \ref{recipe1}.) In this case, the natural homomorphism in $\sVec$ from 
$\Der(\cC)^{ss}$ into $\Der(\cC^{ss})$ is surjective.

\item In case (4) of subsection \ref{ss:semiCayley}, lengthy but straightforward computations give the following Jordan blocks for the action of the nilpotent 
endomorphism $\Delta=\Ad_\sigma-\id$:

\[
\begin{split}
&E_{32} \xrightarrow{\Delta} 
     E_{22}-E_{33}-E_{23}+E_{13}-\ad_{u_1}-\ad_{v_2}+\ad_{v_3}\\[-2pt]
  &\hspace*{160pt}\xrightarrow{\Delta} E_{23}+\ad_{u_1}-\ad_{v_3} \xrightarrow{\Delta} 0,\\
& E_{31} \xrightarrow{\Delta} E_{21}-E_{23}+\ad_{u_2}-\ad_{v_1}
   \xrightarrow{\Delta} -\ad_{u_2} \xrightarrow{\Delta} 0,\\
&\ad_{u_3} \xrightarrow{\Delta} I_3 +\ad_{u_2}+\ad_{v_3}
   \xrightarrow{\Delta} -\ad_{v_3} \xrightarrow{\Delta} 0,\\
& E_{12} \xrightarrow{\Delta} -E_{13} \xrightarrow{\Delta} 0,\\
& \ad_{v_2} \xrightarrow{\Delta} -\ad_{v_3}-\ad_{u_1} \xrightarrow{\Delta} 0,\\
& E_{21} \xrightarrow{\Delta} 0.
\end{split}
\]
Therefore we get a splitting $\Der(\cC)=D\subo\oplus D\subuno\oplus
\Delta(D\subuno)\oplus D_2$ as in \eqref{eq:C0C1C1C2}, with $D\subo=\FF E_{21}$, 
$D\subuno=\FF E_{12}+\FF \ad_{v_2}$, $\Delta(D\subuno)=\FF E_{13}+\FF\ad_{u_1+v_3}$, and $D_2$ the linear span of $\ad_{u_2},\ad_{v_3},\ad_{u_3},I_3,E_{31},E_{32},
E_{23}+\ad_{u_1},E_{21}-E_{23}-\ad_{v_1},(E_{22}-E_{33})+E_{13}-\ad_{v_2}$.
 In consequence, we may take $\Der(\cC)^{ss}=D\subo\oplus D\subuno$, and
Recipe \ref{recipe} gives that  in $\Der(\cC)^{ss}$, 
\[
\begin{split}
\null\qquad\qquad [E_{21},E_{12}]&=\mathrm{proj}_{D\subuno}(E_{22}-E_{11})=
      -\mathrm{proj}_{D\subuno}(I_3+E_{22}-E_{33}),\\
     &= -\mathrm{proj}_{D\subuno}(E_{22}-E_{33}),\quad\text{as $I_3\in D_2$,}\\
     &= \mathrm{proj}_{D\subuno}(E_{13}-\ad_{v_2}),\quad\text{as 
                                    $(E_{22}-E_{33})+E_{13}-\ad_{v_2}\in D_2$,}\\
     & =-\ad_{v_2},\quad\text{as $E_{13}\in \Delta(D\subuno)$,}
\end{split}
\]
and all the other brackets are trivial. 

The action of $\Der(\cC)^{ss}$ on $\cC^{ss}=\FF 1+\FF v_1+\FF v_2$ is given by Recipe \ref{recipe1}: $E_{21}.v_2=-v_1$, and all the other products are trivial.

In this case, the kernel of the natural homomorphism in $\sVec$ from 
$\Der(\cC)^{ss}$ into $\Der(\cC^{ss})$ is $D\subuno$, and this homomorphism is neither injective nor surjective.

\end{itemize}

\bigskip

\section{The extended Freudenthal Magic Square}\label{se:extended}

Assume for a while that the characteristic of the ground field $\FF$ is just different from $2$.

Different authors \cite{Vinberg,BS,LM} have considered several symmetric constructions of Freudenthal's Magic square in terms of two unital composition algebras. We will follow here \cite{EldRMI}, but restricted, for simplicity,
to the use of the so-called \emph{para-Hurwitz algebras}. Let $(\cC,\mu,\boldnup)$
and $(\cC',\mu',\boldnup')$ be two unital composition algebras over a field $\FF$
of characteristic not $2$. Denote in both cases the multiplication by juxtaposition, 
and consider the new multiplications $\bar\mu$ given by $\bar\mu(x,y)=x\bullet y\bydef \overline{x}\,\overline{y}$, where $\overline{x}=\boldnup(x,1)-x$ is the standard conjugation. Define similarly $\bar\mu'$. Consider the associated 
\emph{triality Lie algebras}:
\[
\tri(\cC,\bullet,\boldnup)\bydef\{(d_0,d_1,d_2)\in\frso(\cC,\boldnup)^3\mid
d_0(x\bullet y)=d_1(x)\bullet y+x\bullet d_2(y)\ \forall x,y\in\cC\}
\]
and similarly for $\tri(\cC',\bullet,\boldnup')$. These are Lie algebras
with bracket given componentwise, satisfying that the cyclic permutation
 \begin{equation}\label{eq:theta}
\theta:(d_0,d_1,d_2)\mapsto (d_2,d_0,d_1)
\end{equation}%
is a Lie algebra automorphism (\emph{triality automorphism}). Denote by $\theta'$ the corresponding automorphism of $\tri(\cC',\bullet,\boldnup')$. If $\cC$ is a Cayley algebra, then the projection of $\tri(\cC,\bullet,\boldnup)$ on any of its components gives an isomorphism 
$\tri(\cC,\bullet,\boldnup)\simeq \frso(\cC,\boldnup)$.

For simplicity, we will just write $\tri(\cC)$ and $\tri(\cC')$.

The vector space
\[
\frg=\frg(\cC,\cC')=\bigl(\tri(\cC)\oplus
\tri(\cC')\bigr)\oplus\Bigl(\oplus_{i=0}^2\iota_i(\cC\otimes \cC')\Bigr),
\]
where $\iota_i(\cC\otimes \cC')$ is just a copy of $\cC\otimes \cC'$ ($i=0,1,2$) becomes a Lie algebra with the bracket defined as follows:

\begin{itemize}
\item the Lie bracket in $\tri(\cC)\oplus\tri(\cC')$, which thus becomes  a Lie subalgebra of $\frg$,

\item $[(d_0,d_1,d_2),\iota_i(x\otimes
 x')]=\iota_i\bigl(d_i(x)\otimes x'\bigr)$,

\item
 $[(d_0',d_1',d_2'),\iota_i(x\otimes
 x')]=\iota_i\bigl(x\otimes d_i'(x')\bigr)$,

\item $[\iota_i(x\otimes x'),\iota_{i+1}(y\otimes y')]=
 \iota_{i+2}\bigl((x\bullet y)\otimes (x'\bullet y')\bigr)$ (indices modulo
 $3$),

\item $[\iota_i(x\otimes x'),\iota_i(y\otimes y')]=
 \boldnup'(x',y')\theta^i(t_{x,y})+
 \boldnup(x,y)\theta'^i(t'_{x',y'})\in\tri(\cC)\oplus\tri(\cC')$,
\end{itemize}
for $(d_0,d_1,d_2)\in\tri(\cC)$, $(d_0',d_1',d_2')\in\tri(\cC')$, $x,y\in\cC$,
and $x',y'\in\cC'$,
where $t_{x,y}$ is the element of $\tri(\cC)$ defined as follows:
\begin{equation}\label{eq:txy}
t_{x,y}\bydef\Bigl(s_{x,y},
\frac{1}{2}\bigl(r_yl_x-r_xl_y\bigr),\frac{1}{2}\bigl(l_yr_x-l_xr_y\bigr)\Bigr)
\end{equation}
with $s_{x,y}:z\mapsto \boldnup(x,z)y-\boldnup(y,z)x$, $l_x:z\mapsto x\bullet z$, and $r_x:z\mapsto z\bullet x$, and similarly for $\cC'$.

The Lie algebras thus obtained are semisimple (simple in most cases) and, if the characteristic of the ground field $\FF$ is neither $2$ nor $3$ then the type of the Lie algebras obtained is given by Freudenthal Magic Square, where the index over each row (respectively column) is the dimension of $\cC$ (resp. $\cC'$):
\[
\vbox{\offinterlineskip
 \halign{\hfil\ $#$\quad \hfil&%
 \vreglon #%
 &\hfil\quad $#$\quad \hfil&\hfil$#$\hfil
 &\hfil\quad $#$\quad \hfil&\hfil\quad $#$\quad \hfil\cr
 \bigstrut &width 0pt&1&2&4&8\cr
 &\multispan5{\hreglonfill}\cr
 1&&A_1&A_2&C_3&F_4\cr
 \bigstrut 2&&A_2 &A_2\oplus A_2&A_5&E_6\cr
 \bigstrut 4&&C_3&A_5 &D_6&E_7\cr
 \bigstrut 8&&F_4&E_6& E_7&E_8\cr}}
\]
If the characteristic of the ground field $\FF$ is $3$, instead of simple Lie algebras of type $A_2$ or $A_5$ we obtain forms of the projective general Lie algebra $\frpgl_3(\FF)$ or $\frpgl_6(\FF)$, and instead of simple Lie algebras of type $E_6$, we obtain Lie algebras of dimension $78$ whose derived ideal is simple of type $E_6$ (the simple Lie algebra of type $E_6$ has dimension $77$ in characteristic $3$).

 If $(\cC,\mu,\boldnup)$ is a Cayley algebra, then the projection 
$\pi_0:(d_0,d_1,d_2)\mapsto d_0$, gives a Lie algebra isomorphism $\tri(\cC,\bullet,\boldnup)\simeq\frso(\cC,\boldnup)$. In other words, for any $d_0\in\frso(\cC,\boldnup)$ there are unique $d_1,d_2\in\frso(\cC,\boldnup)$ such that $(d_0,d_1,d_2)$ lies in $\tri(\cC,\bullet,\boldnup)$. Hence the triples
$t_{x,y}$ in \eqref{eq:txy} span the triality Lie algebra $\tri(\cC,\bullet,\boldnup)$. 

Therefore,  
the linear map $\vartheta:\frso(\cC,\boldnup)\rightarrow \frso(\cC,\boldnup)$,
given by 
\begin{equation}\label{eq:vartheta}
\vartheta(s_{x,y})=\frac{1}{2}\bigl(l_yr_x-l_xr_y\bigr),
\end{equation}
is a Lie algebra automorphism that makes
the following diagram commutative ($\theta$ as in \eqref{eq:theta}):
\begin{equation}\label{eq:tri_theta}
\begin{tikzcd}
\tri(\cC,\bullet,\boldnup)\arrow[r, "\theta"]\arrow[d,"\pi_0"']&\tri(\cC,\bullet,\boldnup)\arrow[d,"\pi_0"]\\
\frso(\cC,\boldnup)\arrow[r,"\vartheta"]&\frso(\cC,\boldnup).
\end{tikzcd}
\end{equation}
Hence we also have
\begin{equation}\label{eq:vartheta2}
\vartheta^2(s_{x,y})=\frac{1}{2}\bigl(r_yl_x-r_xl_y\bigr).
\end{equation}

The natural and the two half-spin actions of $\frso(\cC,\boldnup)$ are involved in the Lie bracket of 
$\frg(\cC,\cC')$. The natural action $\Phi_0$ is given by the composition
\[
\frso(\cC,\boldnup)\otimes\cC\xrightarrow{\sim\otimes\id_\cC}\Skew^2(\cC,\boldnup)\otimes\cC
\xrightarrow{\id_{\cC}\otimes\boldnup}\cC,
\]
where $\frso(\cC,\boldnup)$ is identified with $\Skew^2(\cC,\boldnup)$ as in Section \ref{ss:skew}. This composition behaves as follows:
\[
s_{x,y}\otimes z\mapsto (-x\otimes y+y\otimes x)\otimes z
\mapsto -\boldnup(y,z)x+\boldnup(x,z)y= s_{x,y}(z).
\]
The two half-spin representations $\Phi_1$ and $\Phi_2$ are respectively the compositions:
\[
\frso(\cC,\boldnup)\otimes\cC\xrightarrow{\sim\otimes\id_\cC}\Skew^2(\cC,\boldnup)\otimes\cC
\xrightarrow{\id_\cC\otimes\text{`swap'}} \cC\otimes\cC\otimes\cC
\xrightarrow{-\frac{1}{2}\bar\mu\circ(\bar\mu\otimes \id_\cC)}\cC,
\]
given by
\begin{multline*}
s_{x,y}\otimes z\mapsto(-x\otimes y+y\otimes x)\otimes z
\mapsto -x\otimes z\otimes y+y\otimes z\otimes x\\
\mapsto \frac{1}{2}\bigl((x\bullet z)\bullet y-(y\bullet z)\bullet x\bigr)
=\frac{1}{2}\bigl(r_yl_x-r_xl_y\bigr)(z),
\end{multline*}
and 
\[
\frso(\cC,\boldnup)\otimes\cC\xrightarrow{\sim\otimes\id_\cC}\Skew^2(\cC,\boldnup)\otimes\cC
\xrightarrow{\id_\cC\otimes\text{`swap'}} \cC\otimes\cC\otimes\cC
\xrightarrow{\frac{1}{2}\bar\mu\circ(\id_{\cC}\otimes \bar\mu)}\cC,
\]
given by
\begin{multline*}
s_{x,y}\otimes z\mapsto(-x\otimes y+y\otimes x)\otimes z
\mapsto -x\otimes z\otimes y+y\otimes z\otimes x\\
\mapsto \frac{1}{2}\bigl(-x\bullet (z\bullet y)+y\bullet (z\bullet x)\bigr)
=\frac{1}{2}\bigl(l_yr_x-l_xr_y\bigr)(z).
\end{multline*}

The commutativity of \eqref{eq:tri_theta} is then equivalent to the commutativity of the following diagram:
\begin{equation}\label{eq:Phi_vartheta}
\begin{tikzcd}
\frso(\cC,\boldnup)\otimes\cC \arrow[r, "\Phi_0"] \arrow[d, "\vartheta^i\otimes\id_\cC"'] 
& \cC\arrow[d, "\id_\cC"] \\
\frso(\cC,\boldnup)\otimes\cC \arrow[r, "\Phi_i"]  
& \cC
\end{tikzcd}
\end{equation}
for $i=1,2$. Note that all the homomorphisms above are given in terms of the norm $\boldnup$, the multiplication $\bar\mu$ and the braiding (the `swap').

This symmetric construction of Freudenthal's Magic Square was extended, over fields of characteristic $3$, by using the unital composition superalgebras $B(4,2)$ and $B(1,2)$ in \cite{CunhaElduque}, thus obtaining an extended Freudenthal's Magic Square that includes many of the exceptional contragredient simple Lie superalgebras in characteristic $3$. As before, in the second row or column, the superalgebras obtained are no longer simple, but their derived subalgebras are simple.

All these Lie superalgebras have been obtained by Kannan \cite{Kannan} by considering nilpotent derivations of degree $3$ of some of the simple exceptional Lie algebras, and thus looking at these as Lie algebras in the category $\Repa$, whose semisimplification is again $\Ver$.

 Actually, the semisimplification of Cayley algebras in Section \ref{se:octonions} provides a
bridge between the symmetric construction of Freudenthal's Magic Square and the extended square in \cite{CunhaElduque}. 

Assume from now on that the characteristic of our ground field is $3$.

Any order $3$ automorphism of a unital composition algebra $(\cC,\mu,\boldnup)$ is also an automorphism of its para-Hurwitz counterpart, and then it induces an order $3$ automorphism of $\frso(\cC,\boldnup)$ and of $\tri(\cC,\bullet,\boldnup)$ commuting with the triality automorphism.

Therefore, starting with an order $3$ automorphism $\sigma$ of a Cayley algebra 
$(\cC,\mu,\boldnup)$ such that its semisimplification is isomorphic to either $B(1,2)$ or $B(4,2)$, there is an order $3$ automorphism induced in 
$\frg(\cC,\cC')$, where we combine the order $3$ automorphism on $\cC$ and the identity automorphism in $\cC'$. The action of this order $3$ automorphism is as follows:
\begin{itemize}
\item $(d_0,d_1,d_2)\mapsto \bigl(\Ad_\sigma(d_0),\Ad_\sigma(d_1),\Ad_\sigma(d_2)\bigr)$, for $(d_0,d_1,d_2)\in\tri(\cC,\bullet,\boldnup)$, where $\Ad_\sigma(d)=\sigma\circ d\circ\sigma^{-1}$,
\item the action on $\tri(\cC',\bullet,\boldnup')$ is trivial,
\item $\iota_i(x\otimes x')\mapsto \iota_i\bigl(\sigma(x)\otimes x'\bigr)$ for any $i=0,1,2$, 
$x\in\cC$, and $x'\in\cC'$.
\end{itemize}
This allows us to consider $\frg(\cC,\cC')$ as a Lie algebra in $\Repe$.

As any automorphism of $(\cC,\mu,\boldnup)$ commutes with the standard conjugation $x\mapsto
\bar x=\boldnup(1,x)-x$, it turns out that the semisimplification of $(\cC,\bar\mu,\boldnup)$ is the para-Hurwitz superalgebra $(\cC^{ss},\overline{\mu^{ss}},\boldnup^{ss})$. For these superalgebras, the projection $\pi_0:(d_0,d_1,d_2)\mapsto d_0$ is a Lie superalgebra isomorphism 
(see \cite{EldOkubo}). Using Proposition \ref{pr:endos_skew}, we get a chain of isomorphisms of Lie superalgebras:
\[
\tri(\cC,\bar\mu,\boldnup)^{ss}\simeq \frso(\cC,\boldnup)^{ss}\simeq 
\frosp(\cC^{ss},\boldnup^{ss})\simeq\tri(\cC^{ss},\overline{\mu^{ss}},\boldnup^{ss}).
\]
The commutativity of \eqref{eq:tri_theta} shows that, under these isomorphisms, the Lie 
superalgebra automorphism $\theta^{ss}$ of $\tri(\cC^{ss},\overline{\mu^{ss}},\boldnup^{ss})$ 
corresponds to the automorphism $\vartheta^{ss}$ of $\frosp(\cC^{ss},\boldnup^{ss})$, and the 
commutativity of \eqref{eq:Phi_vartheta}, together with the fact that the $\Phi_i$'s are defined in 
terms of $\boldnup$, $\mu$, and the braiding, shows that $\vartheta^{ss}$ satisfies the `super' 
versions of \eqref{eq:vartheta} and \eqref{eq:vartheta2}. But 
$\tri(\cC^{ss},\overline{\mu^{ss}},\boldnup^{ss})$ is spanned by the `super' versions of the triples $t_{x,y}$ in \eqref{eq:txy} (see \cite{EldOkubo}). It follows that, under the isomorphisms above, the automorphism $\theta^{ss}$ of $\tri(\cC,\bar\mu,\boldnup)^{ss}$ corresponds to  the cyclic permutation 
$(d_0,d_1,d_2)\mapsto (d_2,d_0,d_1)$ in $\tri(\cC^{ss},\overline{\mu^{ss}},\boldnup^{ss})$,
 and, as a consequence, that the Lie superalgebra
$\frg(\cC,\cC')^{ss}$ is isomorphic to the superalgebra $\frg(\cC^{ss},\cC')$ in \cite{CunhaElduque}.

The same arguments work if both $\cC$ and $\cC'$ are Cayley algebras endowed with order $3$ automorphisms. We also get an induced order $3$ automorphism of 
$\frg(\cC,\cC')$. These order $3$ automorphisms allow us to see $\frg(\cC,\cC')$ 
as a Lie algebra in $\Repe$ and get its semisimplification $\frg(\cC,\cC')^{ss}$, which is isomorphic to $\frg(\cC^{ss},\cC'^{ss})$.

In other words, the Lie superalgebras in the extended Freudenthal's Magic square
can be obtained by semisimplification of the Lie algebras (in $\Repe$) in the fourth row of the classical Freudenthal's Magic Square in characteristic $3$. 

It must be pointed out here that in \cite{Kannan}, $\frg\bigl(B(1,2),B(1,2)\bigr)$ is obtained from the exceptional Lie algebra of type $E_6$ endowed with a suitable nilpotent derivation of order $3$, while the above comments show that 
$\frg\bigl(B(1,2),B(1,2)\bigr)$ is obtained too from the exceptional Lie algebra of type $E_8$, that is, from the Lie algebra $\frg(\cC,\cC')$ where both $\cC$ and 
$\cC'$ are the split Cayley algebras, endowed with automorphisms of types (3) or (4) in Theorem \ref{th:order3}.


\end{document}